\documentclass{tac}

\usepackage{amsmath}
\usepackage{amssymb}
\usepackage{enumitem}

\usepackage[all,cmtip]{xy}

\input diagxy

\usepackage[colorlinks=true]{hyperref}
\hypersetup{allcolors=[rgb]{0.1,0.1,0.4}}

\author{Hongliang Lai and Lili Shen}

\thanks{The authors acknowledge the support of National Natural Science Foundation of China (11771310 and 11701396) and the Fundamental Research Funds for the Central Universities (YJ201644), as well as earlier, Natural Sciences and Engineering Research Council of Canada (Discovery Grant 501260 held by Professor Walter Tholen). This work was initiated while the first author held a Visiting Professorship and the second author held a Post-Doctoral Fellowship in the Department of Mathematics and Statistics at York University, with the kind support of Professor Walter Tholen.}

\address{School of Mathematics, Sichuan University\\
 Chengdu 610064, China
}

\title{Regularity vs. constructive complete (co)distributivity}

\copyrightyear{2018}


\eaddress{hllai@scu.edu.cn \CR shenlili@scu.edu.cn}

\newtheorem{thm}{Theorem}

\newtheorem{lem}{Lemma}
\newtheorem{prop}{Proposition}
\newtheorem{cor}{Corollary}

\newtheoremrm{rem}{Remark}
\newtheoremrm{defn}{Definition}
\newtheoremrm{exmp}{Example}

\mathrmdef{Hom}
\mathbfdef{Set}

\DeclareMathOperator{\dom}{dom}
\DeclareMathOperator{\cod}{cod}

\DeclareMathOperator{\ob}{ob}

\def\oto{{\bfig\morphism<180,0>[\mkern-4mu`\mkern-4mu;]\place(86,0)[\circ]\efig}}

\def\nra{\relbar\joinrel\joinrel\mapstochar\joinrel\joinrel\rightarrow}
\def\nla{\leftarrow\joinrel\joinrel\joinrel\mapstochar\joinrel\relbar}

\newcommand{\da}{\downarrow}

\newcommand{\ra}{\rightarrow}
\newcommand{\la}{\leftarrow}
\newcommand{\lra}{\longrightarrow}
\newcommand{\lda}{\swarrow}
\newcommand{\rda}{\searrow}

\newcommand{\Lra}{\Longrightarrow}

\newcommand{\rat}{\rightarrowtail}
\newcommand{\bv}{\bigvee}
\newcommand{\bw}{\bigwedge}

\newcommand{\dv}{\dashv}

\newcommand{\nat}{\natural}

\renewcommand{\phi}{\varphi}

\newcommand{\ga}{\gamma}

\newcommand{\Lam}{\Lambda}
\newcommand{\lam}{\lambda}

\newcommand{\CC}{\mathcal{C}}

\newcommand{\CK}{\mathcal{K}}

\newcommand{\CP}{\mathcal{P}}
\newcommand{\CQ}{\mathcal{Q}}

\newcommand{\sV}{{\sf V}}

\newcommand{\sY}{{\sf Y}}

\newcommand{\Fix}{{\sf Fix}}
\newcommand{\sIm}{{\sf Im}}

\newcommand{\bbA}{\mathbb{A}}
\newcommand{\bbB}{\mathbb{B}}

\newcommand{\bbT}{\mathbb{T}}

\newcommand{\BD}{{\bf D}}

\newcommand{\FK}{\mathfrak{K}}

\newcommand{\Arr}{{\bf Arr}}
\newcommand{\Cat}{{\bf Cat}}

\newcommand{\CCat}{{\bf CCat}}
\newcommand{\Chu}{{\bf Chu}}

\newcommand{\Dist}{{\bf Dist}}
\newcommand{\Inf}{{\bf Inf}}
\newcommand{\Info}{{\bf Info}}

\newcommand{\Ord}{{\bf Ord}}

\newcommand{\Rel}{{\bf Rel}}
\newcommand{\Sup}{{\bf Sup}}
\newcommand{\QCat}{\CQ\text{-}\Cat}

\newcommand{\QCCat}{\CQ\text{-}\CCat}

\newcommand{\QDist}{\CQ\text{-}\Dist}
\newcommand{\QChu}{\CQ\text{-}\Chu}
\newcommand{\QInf}{\CQ\text{-}\Inf}
\newcommand{\QInfo}{\CQ\text{-}\Info}

\newcommand{\QSup}{\CQ\text{-}\Sup}

\newcommand{\hga}{\widehat{\ga}}

\newcommand{\hphi}{\widehat{\phi}}
\newcommand{\tphi}{\widetilde{\phi}}

\newcommand{\txi}{\widetilde{\xi}}

\newcommand{\CPd}{\CP^{\dag}}

\newcommand{\sYd}{\sY^{\dag}}
\newcommand{\co}{{\rm co}}
\newcommand{\op}{{\rm op}}
\newcommand{\PA}{\CP\bbA}
\newcommand{\PB}{\CP\bbB}

\newcommand{\PdA}{\CPd\bbA}
\newcommand{\PdB}{\CPd\bbB}

\newcommand{\Kphi}{\CK\phi}
\newcommand{\Kpsi}{\CK\psi}

\newcommand{\DQ}{\BD(\CQ)}
\newcommand{\ArrQ}{\Arr(\CQ)}
\newcommand{\QCD}{(\QSup)_{\rm ccd}}
\newcommand{\of}{\overleftarrow{f}}

\newcommand{\ophi}{\overleftarrow{\phi}}
\newcommand{\opsi}{\overleftarrow{\psi}}

\newcommand{\IdmQ}{\DQ_{\rm idm}}
\newcommand{\RQ}{\DQ_{\rm reg}}
\newcommand{\RArrQDist}{\Arr(\QDist)_{\rm reg}}
\newcommand{\RQDist}{\BD(\QDist)_{\rm reg}}
\newcommand{\Rphi}{\FK\phi}

\newcommand{\KD}{\CK_{\rm d}}
\newcommand{\KR}{\CK_{\rm reg}}

\newcommand{\opccd}{\textsuperscript{op}(ccd)}
\newcommand{\DL}{\BD(L)}
\newcommand{\tc}[1]{\textcircled{\scriptsize #1}}

\numberwithin{equation}{section}

\begin{document}

\maketitle
\begin{abstract}
It is well known that a relation $\varphi$ between sets is regular if, and only if, $\mathcal{K}\varphi$ is completely distributive (cd), where $\mathcal{K}\varphi$ is the complete lattice consisting of fixed points of the Kan adjunction induced by $\varphi$. For a small quantaloid $\mathcal{Q}$, we investigate the $\mathcal{Q}$-enriched version of this classical result, i.e., the regularity of $\mathcal{Q}$-distributors versus the constructive complete distributivity (ccd) of $\mathcal{Q}$-categories, and prove that ``the dual of $\mathcal{K}\varphi$ is (ccd) $\implies$ $\varphi$ is regular $\implies$ $\mathcal{K}\varphi$ is (ccd)'' for any $\mathcal{Q}$-distributor $\varphi$. Although the converse implications do not hold in general, in the case that $\mathcal{Q}$ is a commutative integral quantale, we show that these three statements are equivalent for any $\varphi$ if, and only if, $\mathcal{Q}$ is a Girard quantale.
\end{abstract}

\section{Introduction}

The notion of \emph{regularity} was first introduced by von Neumann \cite{Neumann1936} for rings. It was later adapted to the context of semigroups by Green in his influential paper \cite{Green1951}, which initiated the study of \emph{regular semigroups} for decades \cite{Clifford1961,Howie1995}. More generally, one may consider \emph{regular arrows} in an arbitrary category $\CC$; that is, an arrow $f:X\lra Y$ in $\CC$ such that there exists an arrow $g:Y\lra X$ with $f\circ g\circ f=f$ (see \cite[Exercise I.5.7]{MacLane1998}).

\emph{Constructive complete distributivity} (or \emph{(ccd)} for short) was introduced by Fawcett and Wood in \cite{Fawcett1990}. Explicitly, a complete lattice $A$ is (ccd) if $\sup:\CP A\lra A$, the monotone map sending each down set of $A$ (here $\CP A$ denotes the set of down sets of $A$, ordered by inclusion) to its supremum, admits a left adjoint in $\Ord$. It is well known that (ccd) and (cd), i.e., complete distributivity, are equivalent notions if one assumes the axiom of choice \cite{Fawcett1990,Wood2003}. Moreover, as one may describe a (ccd) lattice precisely by the existence of a string of adjunctions
$$T\dv\sup\dv\sY:A\lra\CP A$$
in $\Ord$, where $\sY$ is the ({\bf 2}-enriched) Yoneda embedding that sends each $x\in A$ to the principal down set $\da x$, the notion of (ccd) can be extended to any (locally small) category; see \cite{Lucyshyn-Wright2012,Marmolejo2012,Rosebrugh1997} for discussions of such categories (called \emph{totally distributive} categories there).

The motivation of this paper originates from a famous theorem in the theory of semigroups that reveals the closed relationship between regular relations (i.e., regular arrows in the category $\Rel$ of sets and relations) and (cd) lattices. Explicitly, each relation $\phi:A\oto B$ between sets induces a \emph{Kan adjunction} \cite{Shen2013a}
$$\phi^*\dv\phi_*:{\bf 2}^A\lra{\bf 2}^B$$
between the powersets of $A$ and $B$, with
$$\phi^*V=\{x\in A\mid\exists y\in V:\ x\phi y\}\quad\text{and}\quad\phi_*U=\{y\in B\mid\forall x\in A:\ x\phi y\implies x\in U\}$$
for $V\subseteq B$, $U\subseteq A$, whose fixed points constitute a complete lattice
$$\Kphi:=\Fix(\phi_*\phi^*)=\{V\subseteq B\mid\phi_*\phi^*V=V\}.$$
The following theorem was first discovered by Zarecki{\u\i} in the case $B=A$ \cite{Zareckiui1963} (see also \cite{Bandelt1980,Schein1976,Yang1969} for related discussions), and was extended to arbitrary relations by Xu and Liu \cite{Xu2004a,Xu2004}:

\begin{thm} \label{regular_cd_classical}
A relation $\phi:A\oto B$ between sets is regular if, and only if, $\Kphi$ is (cd).
\end{thm}

Since \emph{distributors} \cite{Benabou1973,Benabou2000,Borceux1994a,Borceux1994b} (also known as \emph{profunctors} or \emph{bimodules}) generalize relations as functors generalize maps, it is natural to consider the possibility of establishing Theorem \ref{regular_cd_classical} in the framework of category theory, with distributors and (ccd) in lieu of relations and (cd), respectively. The aim of this paper is to investigate this problem in a special case, i.e., for distributors between categories enriched in a small quantaloid $\CQ$ \cite{Heymans2010,Rosenthal1996,Shen2014,Stubbe2005,Stubbe2006}, which is interesting enough to reveal that it is a coincidence for Theorem \ref{regular_cd_classical} to have such an elegant form --- its validity relies on the fact that (ccd) and {\opccd}, i.e., constructive complete \emph{co}distributivity, are equivalent notions when $\CQ={\bf 2}$!

For a small quantaloid $\CQ$, a $\CQ$-distributor $\phi:\bbA\oto\bbB$ between $\CQ$-categories may be thought of as a multi-typed and multi-valued relation that respects $\CQ$-categorical structures in its domain and codomain, and \emph{regular} $\CQ$-distributors are precisely regular arrows in the category $\QDist$ of $\CQ$-categories and $\CQ$-distributors. Each $\phi:\bbA\oto\bbB$ induces a Kan adjunction \cite{Shen2013a}
$$\phi^*\dv\phi_*:\PA\lra\PB$$
between the presheaf $\CQ$-categories of $\bbA$ and $\bbB$, whose fixed points constitute a complete $\CQ$-category $\Kphi$. Moreover, A $\CQ$-category $\bbA$ is \emph{(ccd)} if one has a string of adjoint $\CQ$-functors
$$T\dv\sup\dv\sY:\bbA\lra\PA,$$
where $\sY$ is the ($\CQ$-enriched) Yoneda embedding. Dually, $\bbA$ is \emph{{\opccd}} if $\bbA^{\op}$ is a (ccd) $\CQ^{\op}$-category.

With necessary preparations in Sections \ref{Regular_quantaloid} and \ref{QCat_Kan}, we prove in Section \ref{Regularity_ccd} that $\Kphi$ is (ccd) whenever $\phi$ is a regular $\CQ$-distributor (see Theorem \ref{Kphi_ccd}). Furthermore, in Section \ref{RQDist_QCD} we show that Theorem \ref{Kphi_ccd} gives rise to a (dual) equivalence of categories (see Theorem \ref{KR_equiv})
\begin{equation} \label{RQDist_QCD_intro}
\RQDist^{\op}\simeq\QCD.
\end{equation}
Here $\RQDist$ is the full subcategory of $\BD(\QDist)$, the category of diagonals in $\QDist$ (also known as the Freyd completion of $\QDist$, see Grandis \cite{Grandis2000,Grandis2002}), with objects restricting to regular $\CQ$-distributors; while $\QCD$ is the category of (ccd) $\CQ$-categories and $\sup$-preserving $\CQ$-functors. The equivalence \eqref{RQDist_QCD_intro} extends Stubbe's result that the split-idempotent completion of $\QDist$ is dually equivalent to $\QCD$ \cite{Stubbe2007}, whose prototype comes from the work of Rosebrugh and Wood \cite{Rosebrugh1994} when $\CQ={\bf 2}$.

Unfortunately, the converse statement of Theorem \ref{Kphi_ccd} is not true as the counterexample given in \ref{Kphi_ccd_phi_not_regular} shows. In fact, the regularity of $\phi$ necessarily follows if one assumes $\Kphi$ to be {\opccd}! This observation is stated in Theorem \ref{Kphi_coccd_phi_regular}, whose proof is the most challenging one in this paper. Hence, the chain of logic is essentially as follows:
\begin{equation} \label{logic_chain}
\Kphi\ \text{is {\opccd}}\ \Lra\ \phi\ \text{is regular}\ \Lra\ \Kphi\ \text{is (ccd)}.
\end{equation}

Although both the implications in \eqref{logic_chain} are proper when quantified over $\phi$ and $\CQ$ as one could easily see from Examples \ref{Kphi_ccd_phi_not_regular} and \ref{phi_regular_Kphi_not_opccd}, in Section \ref{Girard_quantaloids_reg_ccd_opccd} it is shown that (ccd) and {\opccd} are equivalent notions when $\CQ$ is a \emph{Girard} quantaloid \cite{Rosenthal1992}, which leads to
\begin{equation} \label{logic_chain_Girard}
\Kphi\ \text{is {\opccd}}\iff\phi\ \text{is regular}\iff\Kphi\ \text{is (ccd)}
\end{equation}
in this case (see Theorem \ref{regular_ccd_Girard}). In particular, since {\bf 2} is a Girard quantale (i.e., a one-object Girard quantaloid), Theorem \ref{regular_cd_classical} becomes a special case of Theorem \ref{regular_ccd_Girard}; in this sense we indeed give a new proof for the following version of Theorem \ref{regular_cd_classical} which does not require the axiom of choice:

\begin{thm} \label{regular_cd_classical_ccd}
A relation $\phi:A\oto B$ between sets is regular if, and only if, $\Kphi$ is (ccd).
\end{thm}

Finally, we wish to find the minimal requirement for $\CQ$ to establish the equivalences \eqref{logic_chain_Girard}. Some partial results are obtained in Section \ref{Q_quantale}, where $\CQ$ is assumed to be a commutative integral quantale, and we show that the equivalences \eqref{logic_chain_Girard} hold for any $\phi$ enriched in such $\CQ$ if, and only if, $\CQ$ is a Girard quantale (see Theorem \ref{ccd_coccd_Girard}).

\section{Regular arrows in a quantaloid} \label{Regular_quantaloid}

A \emph{quantaloid} \cite{Rosenthal1996} $\CQ$ is a category enriched in the symmetric monoidal closed category $\Sup$ of complete lattices and join-preserving maps. Explicitly, $\CQ$ is a locally ordered 2-category whose hom-sets are complete lattices such that the composition $\circ$ of $\CQ$-arrows preserves joins on both sides, with the induced adjoints
$$-\circ f\dv -\lda f:\ \CQ(X,Z)\lra\CQ(Y,Z)\quad\text{and}\quad g\circ -\dv g\rda -:\ \CQ(X,Z)\lra\CQ(X,Y)$$
satisfying
$$g\circ f\leq h\iff g\leq h\lda f\iff f\leq g\rda h$$
for all $\CQ$-arrows $f:X\lra Y$, $g:Y\lra Z$, $h:X\lra Z$.

A $\CQ$-arrow $f:X\lra Y$ is \emph{regular} \cite{MacLane1998} if there exists a $\CQ$-arrow $g:Y\lra X$ such that $f\circ g\circ f=f$. Note that for any $\CQ$-arrow $f:X\lra Y$,
$$\of:=(f\rda f)\lda f:Y\lra X$$
is the greatest $\CQ$-arrow $g:Y\lra X$ with $f\circ g\circ f\leq f$. This observation makes it easy to verify the following characterization of regular $\CQ$-arrows:

\begin{prop} \label{regular_condition}
For any $\CQ$-arrow $f$, the following statements are equivalent:
\begin{enumerate}[label={\rm(\roman*)}]
\item $f$ is regular.
\item $f=f\circ\of\circ f$.
\item $f\leq f\circ\of\circ f$.
\end{enumerate}
\end{prop}

\begin{exmp}
\begin{enumerate}[label={\rm(\arabic*)}]
\item Any $\CQ$-arrow $f:X\lra Y$ that admits a left or right adjoint in the 2-category $\CQ$ is regular.
\item The category $\Rel$ is in fact a quantaloid under the inclusion order of relations. For any relation $\phi:X\oto Y$, it is straightforward to check that $\ophi:Y\oto X$ is precisely the relation $\phi^{\leq}$ defined by Ern{\'e} in \cite[Section 4]{Erne1993}; that is, the greatest relation $\psi:Y\oto X$ with $\phi\circ\psi\circ\phi\subseteq\phi$.
\item $\Sup$ is itself a quantaloid, in which a join-preserving map $f:X\lra Y$ between (ccd) lattices is regular if, and only if, the image of $f$, $\sIm f=\{f(x)\mid x\in X\}$, is a (ccd) lattice (see \cite[Theorem 3.1]{Hoehle2011}).
\end{enumerate}
\end{exmp}

Each quantaloid $\CQ$ induces an arrow category $\ArrQ$ of $\CQ$ with $\CQ$-arrows as objects and pairs of $\CQ$-arrows ($u:X_1\lra X_2,\ v:Y_1\lra Y_2$) satisfying
$$g\circ u=v\circ f$$
$$\bfig
\square<500,400>[X_1`X_2`Y_1`Y_2;u`f`g`v]
\efig$$
as arrows from $f:X_1\lra Y_1$ to $g:X_2\lra Y_2$. $\ArrQ$ is again a quantaloid with the componentwise local order inherited from $\CQ$.

For arrows $(u,v),(u',v'):f\lra g$ in $\ArrQ$, denote by $(u,v)\sim(u',v')$ if the commutative squares
$$\bfig
\square[X_1`X_2`Y_1`Y_2;u`f`g`v]
\square(1000,0)[X_1`X_2`Y_1`Y_2;u'`f`g`v']
\morphism(0,500)/-->/<500,-500>[X_1`Y_2;]
\morphism(1000,500)/-->/<500,-500>[X_1`Y_2;]
\efig$$
have the same \emph{diagonal}; that is, if
$$g\circ u=v\circ f=g\circ u'=v'\circ f.$$
``$\sim$'' gives rise to a congruence on $\ArrQ$, and the induced quotient quantaloid, denoted by $\DQ$, is called the quantaloid of \emph{diagonals} in $\CQ$ (see \cite[Example 2.14]{Stubbe2014}). The underlying category of $\DQ$ is also known as the Freyd completion of $\CQ$, which makes sense for an arbitrary category instead of a quantaloid (see \cite{Grandis2000,Grandis2002}).

By restricting the objects of $\DQ$ on regular $\CQ$-arrows one has a full subquantaloid $\RQ$ of $\DQ$. Again, $\RQ$ has a full subquantaloid $\IdmQ$ \cite{Stubbe2005a} whose objects are idempotent $\CQ$-arrows, which is known as the \emph{split-idempotent completion} of $\CQ$.\footnote{Arrows from an idempotent $e:X\lra X$ to an idempotent $f:Y\lra Y$ in $\IdmQ$ can be equivalently described as $\CQ$-arrows $u:X\lra Y$ with $u\circ e=u=f\circ u$.}

\begin{prop} \label{RQ_equiv_IdmQ}
$\RQ$ is equivalent to its full subquantaloid $\IdmQ$.
\end{prop}

\begin{proof}
For any regular $\CQ$-arrow $f:X\lra Y$, it is clear that $\of\circ f:X\lra X$ is an idempotent $\CQ$-arrow. It suffices to verify that $f$ and $\of\circ f$ are isomorphic objects in $\RQ$, which is easy since
$$(\of\circ f,\of):f\lra\of\circ f\quad\text{and}\quad(1_X,f):\of\circ f\lra f$$
satisfy
$$(\of\circ f,\of)\circ(1_X,f)\sim(1_X,1_Y):f\lra f$$
and
$$(1_X,f)\circ(\of\circ f,\of)\sim(1_X,1_X):\of\circ f\lra\of\circ f,$$
establishing an isomorphism in $\RQ$.
\end{proof}

\section{$\CQ$-categories and Kan adjunctions} \label{QCat_Kan}

From now on, let $\CQ$ be a \emph{small} quantaloid. We shall use the same notations of $\CQ$-categories, $\CQ$-distributors and $\CQ$-functors as fixed in \cite[Subsection 3.1]{Shen2016a}, and here is a brief summary of the basic notions.

A (small) \emph{$\CQ$-category} $\bbA$ is determined by a set $\bbA_0$ of objects, a \emph{type} map $t:\bbA_0\lra\ob\CQ$, and hom-arrows $\bbA(x,y)\in\CQ(tx,ty)$ with $1_{tx}\leq\bbA(x,x)$ and $\bbA(y,z)\circ\bbA(x,y)\leq\bbA(x,z)$ for all $x,y,z\in\bbA_0$.

A \emph{$\CQ$-distributor} $\phi:\bbA\oto\bbB$ between $\CQ$-categories is a map that assigns to each pair $(x,y)\in\bbA_0\times\bbB_0$ a $\CQ$-arrow $\phi(x,y)\in\CQ(tx,ty)$, such that $\bbB(y,y')\circ\phi(x,y)\circ\bbA(x',x)\leq\phi(x',y')$ for all $x,x'\in\bbA_0$, $y,y'\in\bbB_0$. $\CQ$-categories and $\CQ$-distributors constitute a quantaloid $\QDist$ with the pointwise local order inherited from $\CQ$.

A \emph{$\CQ$-functor} (resp. \emph{fully faithful $\CQ$-functor}) $F:\bbA\lra\bbB$ between $\CQ$-categories is a map $F:\bbA_0\lra\bbB_0$ with $tx=t(Fx)$ and $\bbA(x,y)\leq\bbB(Fx,Fy)$ (resp. $\bbA(x,y)=\bbB(Fx,Fy)$) for all $x,y\in\bbA_0$. $\CQ$-categories and $\CQ$-functors are organized into a 2-category $\QCat$ with 2-cells given by the pointwise underlying order
\begin{align*}
F\leq G:\bbA\lra\bbB&\iff\forall x\in\bbA_0:\ Fx\leq Gx\ \text{in}\ \bbB_0\\
&\iff\forall x\in\bbA_0:\ 1_{tx}\leq\bbB(Fx,Gx).
\end{align*}
Each $\CQ$-functor $F:\bbA\lra\bbB$ induces an adjunction $F_{\nat}\dv F^{\nat}$ in $\QDist$ with
$$F_{\nat}=\bbB(F-,-):\bbA\oto\bbB,\quad\text{and}\quad F^{\nat}=\bbB(-,F-):\bbB\oto\bbA,$$
called respectively the \emph{graph} and \emph{cograph} of $F$, which are both 2-functorial as
$$(-)_{\nat}:\QCat\lra(\QDist)^{\co},\quad(-)^{\nat}:\QCat\lra(\QDist)^{\op},$$
where ``$\co$'' refers to the dualization of 2-cells.

A \emph{presheaf} with type $X$ on a $\CQ$-category $\bbA$ is a $\CQ$-distributor $\mu:\bbA\oto\star_X$, where $\star_X$ is the $\CQ$-category with only one object of type $X$. Presheaves on $\bbA$ constitute a $\CQ$-category $\PA$ with $\PA(\mu,\mu')=\mu'\lda\mu$ for all $\mu,\mu'\in\PA$. Dually, the $\CQ$-category $\PdA$ of \emph{copresheaves} on $\bbA$ consists of $\CQ$-distributors $\lam:\star_X\oto\bbA$ as objects with type $X$ and $\PdA(\lam,\lam')=\lam'\rda\lam$ for all $\lam,\lam'\in\PdA$.

\begin{rem} \label{PdA_QDist_order}
For any $\CQ$-category $\bbA$, it follows from the definition that the underlying order on $\PA$ coincides with the local order in $\QDist$, while the underlying order on $\PdA$ is the \emph{reverse} local order in $\QDist$, i.e.,
$$\lam\leq\lam'\ \text{in}\ \PdA\iff\lam'\leq\lam\ \text{in}\ \QDist.$$
In order to get rid of the confusion about the symbol $\leq$, we make the convention that the symbol $\leq$ between $\CQ$-distributors always refers to the local order in
$\QDist$ unless otherwise specified.
\end{rem}


A $\CQ$-category $\bbA$ is \emph{complete} if each $\mu\in\PA$ has a \emph{supremum} $\sup\mu\in\bbA_0$ of type $t\mu$ such that
\begin{equation} \label{sup_def}
\bbA(\sup\mu,-)=\bbA\lda\mu;
\end{equation}
or equivalently, if the \emph{Yoneda embedding} $\sY:\bbA\lra\PA$ has a left adjoint $\sup:\PA\lra\bbA$ in $\QCat$. It is well known that $\bbA$ is a complete $\CQ$-category if, and only if, $\bbA^{\op}$ is a complete $\CQ^{\op}$-category\footnote{The dual of a $\CQ$-category $\bbA$, denoted by $\bbA^{\op}$, is a $\CQ^{\op}$-category with $\bbA^{\op}_0=\bbA_0$ and $\bbA^{\op}(x,y)=\bbA(y,x)$ for all $x,y\in\bbA_0$.} \cite{Stubbe2005}, where the completeness of $\bbA^{\op}$ may be translated as each $\lam\in\PdA$ admitting an \emph{infimum} $\inf\lam\in\bbA_0$ of type $t\lam$ such that
\begin{equation} \label{inf_def}
\bbA(-,\inf\lam)=\lam\rda\bbA;
\end{equation}
or equivalently, the \emph{co-Yoneda embedding} $\sYd:\bbA\oto\PdA,\ x\mapsto\bbA(x,-)$ admitting a right adjoint $\inf:\PdA\lra\bbA$ in $\QCat$.

\begin{lem}[Yoneda] (See \cite{Stubbe2005}.) \label{Yoneda_lemma}
For any $\CQ$-category $\bbA$ and $\mu\in\PA$, $\lam\in\PdA$,
$$\mu=\PA(\sY_{\bbA}-,\mu)=(\sY_{\bbA})_{\nat}(-,\mu),\quad\lam=\PdA(\lam,\sYd_{\bbA}-)=(\sYd_{\bbA})^{\nat}(\lam,-).$$
\end{lem}

\begin{thm} (See \cite{Stubbe2006}.) \label{complete_tensor}
A $\CQ$-category $\bbA$ is complete if, and only if,
\begin{enumerate}[label={\rm(\arabic*)}]
\item $\bbA$ is \emph{tensored} in the sense that for any $x\in\bbA_0$ and $f\in\CP(tx)$\footnote{$f\in\CP(tx):=\CP\star_{tx}$ is essentially a $\CQ$-arrow with domain $tx$. Similarly, $g\in\CPd(tx):=\CPd\star_{tx}$ is precisely a $\CQ$-arrow with codomain $tx$.}, there exists $f\otimes x\in\bbA_0$ of type $\cod f$ with $\bbA(f\otimes x,-)=\bbA(x,-)\lda f$;
\item $\bbA$ is \emph{cotensored} in the sense that the for any $x\in\bbA_0$ and $g\in\CPd(tx)$, there exists $g\rat x\in\bbA_0$ of type $\dom g$ with $\bbA(-,g\rat x)=g\rda\bbA(-,x)$;
\item $\bbA$ is \emph{order-complete} in the sense that each $\bbA_X$, the $\CQ$-subcategory of $\bbA$ consisting of all objects of type $X\in\ob\CQ$, admits all joins in the underlying order.
\end{enumerate}
\end{thm}

\begin{prop} (See \cite{Shen2013a,Stubbe2005}.) \label{PA_PdA_sup}
For any $\CQ$-category $\bbA$, $\PA$ and $\PdA$ are both complete $\CQ$-categories in which
\begin{enumerate}[label={\rm(\arabic*)}]
\item $f\otimes_{\PA}\mu=f\circ\mu$,\quad $g\rat_{\PA}\mu=g\rda\mu$ for all $\mu\in\PA$, $f\in\CP(t\mu)$, $g\in\CPd(t\mu)$,
\item $f\otimes_{\PdA}\lam=\lam\lda f$,\quad $g\rat_{\PdA}\lam=\lam\circ g$ for all $\lam\in\PdA$, $f\in\CP(t\lam)$, $g\in\CPd(t\lam)$,
\item ${\sup}_{\PA}\Theta=\Theta\circ(\sY_{\bbA})_{\nat}=\displaystyle\bv\limits_{\mu\in\PA}\Theta(\mu)\circ\mu$ for all $\Theta\in\CP\PA$,
\item ${\inf}_{\PA}\Lam=\Lam\rda(\sY_{\bbA})_{\nat}=\displaystyle\bw\limits_{\mu\in\PA}\Lam(\mu)\rda\mu$ for all $\Lam\in\CPd\PA$,
\item ${\sup}_{\PdA}\Theta=(\sYd_{\bbA})^{\nat}\lda\Theta=\displaystyle\bw\limits_{\lam\in\PdA}\lam\lda\Theta(\lam)$ for all $\Theta\in\CP\PdA$, and
\item ${\inf}_{\PdA}\Lam=(\sYd_{\bbA})^{\nat}\circ\Lam=\displaystyle\bv\limits_{\lam\in\PdA}\lam\circ\Lam(\lam)$ for all $\Lam\in\CPd\PdA$.
\end{enumerate}
\end{prop}


\begin{prop} \label{monad_reflective} (See \cite{Shen2013a}.)
Suppose  $F:\bbA\lra\bbA$ is a $\CQ$-monad (resp. $\CQ$-comonad) on a skeletal\footnote{A $\CQ$-category $\bbA$ is \emph{skeletal} if $x\cong y$ (i.e., $x\leq y$ and $y\leq x$) in the underlying order of $\bbA$ necessarily forces $x=y$.} $\CQ$-category $\bbA$; that is, $1_{\bbA}\leq F$ (resp. $F\leq 1_{\bbA}$) and $FF=F$. Let
$$\Fix(F):=\{x\in\bbA_0\mid Fx=x\}=\{Fx\mid x\in\bbA_0\}$$
be the $\CQ$-subcategory of $\bbA$ consisting of the fixed points of $F$. Then
\begin{enumerate}[label={\rm(\arabic*)}]
\item the inclusion $\CQ$-functor $\Fix(F)\ \to/^(->/\bbA$ is right (resp. left) adjoint to the codomain restriction $F:\bbA\lra\Fix(F)$;
\item $\Fix(F)$ is a complete $\CQ$-category provided so is $\bbA$.
\end{enumerate}
\end{prop}

If $\bbB=\Fix(F)$ for a $\CQ$-monad $F:\bbA\lra\bbA$, suprema in $\bbB$ are given by ${\sup}_{\bbB}\mu=F{\sup}_{\bbA}(\mu\circ I^{\nat})$ for all $\mu\in\PB$. In particular, for all $x,x_i\in\bbB_0$ $(i\in I)$, $f\in\CP(tx)$,
\begin{equation} \label{tensor_closure_system}
f\otimes_{\bbB}x=F(f\otimes_{\bbA}x),\quad\bigsqcup_{i\in I}x_i=F\Big(\bv_{i\in I}x_i\Big),
\end{equation}
where $\bigsqcup$ and $\bv$ respectively denote the underlying joins in $\bbB$ and $\bbA$.

Each $\CQ$-distributor $\phi:\bbA\oto\bbB$ induces a \emph{Kan adjunction} \cite{Shen2013a} $\phi^*\dv\phi_*$ in $\QCat$ given by
\begin{align*}
&\phi^*:\PB\lra\PA,\quad \lam\mapsto\lam\circ\phi,\\
&\phi_*:\PA\lra\PB,\quad \mu\mapsto\mu\lda\phi
\end{align*}
and a \emph{dual Kan adjunction} \cite{Shen2014} $\phi_{\dag}\dv\phi^{\dag}$ given by
\begin{align*}
&\phi_{\dag}:\PdB\lra\PdA,\quad\lam\mapsto\phi\rda\lam,\\
&\phi^{\dag}:\PdA\lra\PdB,\quad\mu\mapsto\phi\circ\mu.
\end{align*}

The fixed points of the $\CQ$-monad $\phi_*\phi^*:\PB\lra\PB$ and the $\CQ$-comonad $\phi^*\phi_*:\PA\lra\PA$ induced by the Kan adjunction $\phi^*\dv\phi_*:\PA\lra\PB$,
\begin{align*}
&\Kphi:=\Fix(\phi_*\phi^*)=\{\lam\in\PB\mid\phi_*\phi^*\lam=\lam\}\quad\text{and}\\
&\Rphi:=\Fix(\phi^*\phi_*)=\{\mu\in\PA\mid\phi^*\phi_*\mu=\mu\},
\end{align*}
are both complete $\CQ$-categories by Proposition \ref{monad_reflective}(2) since so are $\PB$ and $\PA$. It is obvious that $\Kphi$ and $\Rphi$ are isomorphic $\CQ$-categories with the isomorphisms given by
$$\phi^*:\Kphi\lra\Rphi\quad\text{and}\quad\phi_*:\Rphi\lra\Kphi.$$

\begin{exmp} \label{KA}
For the identity $\CQ$-distributor $\bbA:\bbA\oto\bbA$ on a $\CQ$-category $\bbA$, $\CK\bbA=\FK\bbA=\PA$ is precisely the presheaf $\CQ$-category of $\bbA$.
\end{exmp}

\begin{prop} \label{star_graph_adjoint} (See \cite{Heymans2010}.)
$(-)^*:(\QDist)^{\op}\lra\QCat$ and $(-)^{\dag}:(\QDist)^{\co}\lra\QCat$ are both 2-functorial, and one has two pairs of adjoint 2-functors
$$(-)^{\nat}\dv(-)^*:(\QDist)^{\op}\lra\QCat\quad\text{and}\quad(-)_{\nat}\dv(-)^{\dag}:(\QDist)^{\co}\lra\QCat.$$
\end{prop}

The adjunctions $(-)^{\nat}\dv(-)^*$ and $(-)_{\nat}\dv(-)^{\dag}:(\QDist)^{\co}\lra\QCat$ give rise to isomorphisms
$$(\QCat)^{\co}(\bbA,\PdB)\cong\QDist(\bbA,\bbB)\cong\QCat(\bbB,\PA)$$
for all $\CQ$-categories $\bbA$, $\bbB$. We denote by
\begin{align}
&\tphi:\bbB\lra\PA,\quad\tphi y=\phi(-,y)=\phi^*\sY_{\bbB}y,\label{tphi_def}\\
&\hphi:\bbA\lra\PdB,\quad\hphi x=\phi(x,-)=\phi^{\dag}\sYd_{\bbA}x \label{hphi_def}
\end{align}
for the \emph{transposes} of each $\CQ$-distributor $\phi:\bbA\oto\bbB$, which are determined by
\begin{equation} \label{dist_Yoneda}
\phi=(\sYd_{\bbB})^{\nat}\circ\hphi_{\nat}=\tphi^{\nat}\circ(\sY_{\bbA})_{\nat}.
\end{equation}

Each $\CQ$-functor $F:\bbA\lra\bbB$ gives rise to four $\CQ$-functors between the $\CQ$-categories of presheaves and copresheaves on $\bbA$, $\bbB$:
\begin{equation} \label{Fra_def}
\begin{array}{llll}
F^{\ra}:=(F^{\nat})^*:&\PA\lra\PB,& F^{\la}:=(F_{\nat})^*=(F^{\nat})_*:&\PB\lra\PA,\\
F^{\nra}:=(F_{\nat})^{\dag}:&\PdA\lra\PdB,& F^{\nla}:=(F^{\nat})^{\dag}=(F_{\nat})_{\dag}:&\PdB\lra\PdA.
\end{array}
\end{equation}
As special cases of (dual) Kan adjunctions one immediately has $F^{\ra}\dv F^{\la}$ and $F^{\nla}\dv F^{\nra}$ in $\QCat$.

\begin{prop} \label{la_sup_preserving} (See \cite{Stubbe2005}.)
Let $F:\bbA\lra\bbB$ be a $\CQ$-functor, with $\bbA$ complete. Then the following statements are equivalent:
\begin{enumerate}[label={\rm(\roman*)}]
\item $F$ is a left (resp. right) adjoint in $\QCat$.
\item $F$ is \emph{$\sup$-preserving} (resp. \emph{$\inf$-preserving}) in the sense that $F\sup_{\bbA}=\sup_{\bbB}F^{\ra}$ (resp. $F\inf_{\bbA}=\inf_{\bbB}F^{\nra}$).
\item $F$ is a left (resp. right) adjoint between the underlying ordered sets of $\bbA$, $\bbB$, and preserves tensors (resp. cotensors) in the sense that $F(f\otimes_{\bbA}x)= f\otimes_{\bbB}Fx$ (resp. $F(f\rat_{\bbA} x)=f\rat_{\bbB}Fx$) for all choices of $f$ and $x$.
\end{enumerate}
\end{prop}

Therefore, left adjoint $\CQ$-functors between complete $\CQ$-categories are precisely $\sup$-preserving $\CQ$-functors, and we denote by $\QSup$ the category of skeletal complete $\CQ$-categories and $\sup$-preserving $\CQ$-functors\footnote{$\QSup$ is written as $\QCCat$ in \cite{Shen2014,Shen2016a,Shen2013a}.}, which is in fact a quantaloid with pointwise local order inherited from $\QCat$. Dually, complete $\CQ$-categories and $\inf$-preserving $\CQ$-functors (or equivalently, right adjoint $\CQ$-functors) constitute a 2-subcategory $\QInf$ of $\QCat$; however, it should be careful that $(\QInf)^{\co}$ (rather than $\QInf$ itself) is a quantaloid. It is not difficult to verify the following isomorphisms of quantaloids:

\begin{prop} \label{QSup_QInf}
$\QSup\cong(\QInf)^{\co\op}\cong(\CQ^{\op}\text{-}\Sup)^{\op}\cong(\CQ^{\op}\text{-}\Inf)^{\co}$.
\end{prop}

\section{Regularity implies (ccd)} \label{Regularity_ccd}

$\CQ$-distributor $\phi:\bbA\oto\bbB$ is \emph{regular} if $\phi$ is a regular arrow in the quantaloid $\QDist$; that is, if the $\CQ$-distributor
$$\ophi=(\phi\rda\phi)\lda\phi:\bbB\oto\bbA$$
satisfies $\phi=\phi\circ\ophi\circ\phi$ (see Proposition \ref{regular_condition}).

\begin{exmp}
The graph $F_{\nat}:\bbA\oto\bbB$ and the cograph $F^{\nat}:\bbB\oto\bbA$ of any $\CQ$-functor $F:\bbA\lra\bbB$ are regular $\CQ$-distributors since $F_{\nat}\dv F^{\nat}$ in $\QDist$.
\end{exmp}

A $\CQ$-category $\bbA$ is \emph{constructively completely distributive}, or \emph{(ccd)} for short, if it is complete and $\sup_{\bbA}:\PA\lra\bbA$ admits a left adjoint $T_{\bbA}:\bbA\lra\PA$ in $\QCat$; that is, if there exists a string of adjoint $\CQ$-functors
$$T_{\bbA}\dv{\sup}_{\bbA}\dv\sY_{\bbA}:\bbA\lra\PA.$$

\begin{exmp} \label{PA_ccd}
For any $\CQ$-category $\bbA$, $\PA$ is a (ccd) $\CQ$-category. Indeed, from \eqref{Fra_def} and Proposition \ref{PA_PdA_sup}(3) one sees that $\sup_{\PA}=\sY_{\bbA}^{\la}:\CP\PA\lra\PA$ has a left adjoint $\sY_{\bbA}^{\ra}:\PA\lra\CP\PA$ in $\QCat$.
\end{exmp}

\begin{prop} \label{retract_ccd} (See \cite{Pu2015}.)
A retract\footnote{In any category $\CC$, an object $X$ is a \emph{retract} of an object $Y$ if there are $\CC$-arrows $f:X\lra Y$ and $g:Y\lra X$ such that $g\circ f=1_X$ (see \cite[Definition 1.7.3]{Borceux1994a}).} of a (ccd) $\CQ$-category in $\QSup$ is (ccd).
\end{prop}

\begin{proof}
Let $\bbA$ be a (ccd) $\CQ$-category with $T_A\dv\sup_{\bbA}:\PA\lra\bbA$. If a pair of left adjoint $\CQ$-functors $\bbB\two/->`<-/^F_G\bbA$ satisfies $GF=1_{\bbB}$, in order to prove that $\bbB$ is (ccd), it suffices to verify $G^{\ra}T_{\bbA}F\dv\sup_{\bbB}:\PB\lra\bbB$. This is easy since from Proposition \ref{la_sup_preserving}(ii) one immediately has
\begin{align*}
&{\sup}_{\bbB}G^{\ra}T_{\bbA}F=G{\sup}_{\bbA}T_{\bbA}F\geq GF=1_{\bbB},\\
&G^{\ra}T_{\bbA}F{\sup}_{\bbB}=G^{\ra}T_{\bbA}{\sup}_{\bbA}F^{\ra}\leq G^{\ra}F^{\ra}=(GF)^{\ra}=1_{\PB},
\end{align*}
and the conclusion thus follows.
\end{proof}

\begin{prop} \label{idempotent_ccd}
If $\theta:\bbA\oto\bbA$ is an idempotent $\CQ$-distributor, then $\CK\theta$ is a (ccd) $\CQ$-category.
\end{prop}

\begin{proof}
Since $\FK\theta\cong\CK\theta$, we show that $\FK\theta$ is (ccd). Indeed, the inclusion $\CQ$-functor $J:\FK\theta\lra\PA$ is a left adjoint in $\QCat$ by Proposition \ref{monad_reflective}(1), and so is the codomain restriction $\theta^*:\PA\lra\FK\theta$, whose composition $\theta^*J=1_{\FK\theta}$ since $\theta$ is idempotent. Thus the conclusion follows from Proposition \ref{retract_ccd}.
\end{proof}

\begin{thm} \label{Kphi_ccd}
If $\phi$ is a regular $\CQ$-distributor, then $\Kphi$ is a (ccd) $\CQ$-category. Conversely, every skeletal (ccd) $\CQ$-category is isomorphic to $\Kphi$ for some regular $\CQ$-distributor $\phi$.
\end{thm}

\begin{proof}
Each regular $\CQ$-distributor $\phi:\bbA\oto\bbB$ induces an idempotent $\CQ$-distributor $\phi\circ\ophi:\bbB\oto\bbB$. To show that $\Kphi$ is (ccd), by Proposition \ref{idempotent_ccd} it suffices to prove $\Kphi=\CK(\phi\circ\ophi)$. On one hand, $\lam\in\Kphi$ implies
$$\lam=\phi_*\phi^*\lam=(\phi\circ\ophi\circ\phi)_*\phi^*\lam=(\phi\circ\ophi)_*\phi_*\phi^*\lam=(\phi\circ\ophi)_*\lam\in\CK(\phi\circ\ophi).$$
On the other hand, $\lam\in\CK(\phi\circ\ophi)$ implies $\lam=(\phi\circ\ophi)_*\lam'$ for some $\lam'\in\PB$, and consequently
$$\lam=(\phi\circ\ophi)_*\lam'=\phi_*(\ophi_*\lam')\in\Kphi.$$

Conversely, each skeletal (ccd) $\CQ$-category $\bbA$ induces a $\CQ$-distributor $\theta_{\bbA}:\bbA\oto\bbA$ whose transpose (see \eqref{tphi_def})
$$\widetilde{\theta_{\bbA}}= T_{\bbA}:\bbA\lra\PA$$
is the left adjoint of $\sup_{\bbA}:\PA\lra\bbA$. We show that $\theta_{\bbA}$ is a regular $\CQ$-distributor and $\bbA\cong\CK\theta_{\bbA}$.

First, $\theta_{\bbA}$ is idempotent, thus regular. Indeed,
\begin{align*}
\theta_{\bbA}(-,x)&=T_{\bbA}x&(\text{Equation \eqref{tphi_def}})\\
&=T_{\bbA}{\sup}_{\bbA}T_{\bbA}x&(T_{\bbA}\dv{\sup}_{\bbA})\\
&={\sup}_{\PA}T_{\bbA}^{\ra}T_{\bbA}x&(\text{Proposition \ref{la_sup_preserving}(ii)})\\
&=(T_{\bbA}x)\circ T_{\bbA}^{\nat}\circ(\sY_{\bbA})_{\nat}&(\text{\eqref{Fra_def} and Proposition \ref{PA_PdA_sup}(3)})\\
&=\theta_{\bbA}(-,x)\circ\theta_{\bbA}&(\text{Equation \eqref{dist_Yoneda}})
\end{align*}
for all $x\in\bbA_0$, showing that $\theta_{\bbA}\circ\theta_{\bbA}=\theta_{\bbA}$.

Second, $\bbA\cong\CK\theta_{\bbA}$. Since it is easy to see $\bbA\cong\sIm\sY_{\bbA}$, where $\sIm\sY_{\bbA}=\{\sY_{\bbA}x\mid x\in\bbA_0\}$ is a $\CQ$-subcategory of $\PA$, it remains to prove $\CK\theta_{\bbA}=\sIm\sY_{\bbA}$. On one hand, $\mu\in\CK\theta_{\bbA}$ implies $\mu=\mu'\lda\theta_{\bbA}$ for some $\mu'\in\PA$, and consequently
$$\mu=\mu'\lda\theta_{\bbA}=\PA(T_{\bbA}-,\mu')=\bbA(-,{\sup}_{\bbA}\mu')=\sY_{\bbA}{\sup}_{\bbA}\mu'\in\sIm\sY_{\bbA}.$$
On the other hand, $\sY_{\bbA}x\in\CK\theta_{\bbA}$ for any $x\in\bbA_0$ since $T_{\bbA}\dv\sup_{\bbA}\dv\sY_{\bbA}$ implies
$$1_{\bbA}\leq{\sup}_{\bbA}T_{\bbA}={\sup}_{\bbA}T_{\bbA}{\sup}_{\bbA}\sY_{\bbA}\leq{\sup}_{\bbA}\sY_{\bbA}=1_{\bbA};$$
that is, $1_{\bbA}={\sup}_{\bbA}T_{\bbA}$. It follows that
$$(\theta_{\bbA})_*(\theta_{\bbA})^*\sY_{\bbA}x=\theta_{\bbA}(-,x)\lda\theta_{\bbA}=\PA(T_{\bbA}-,T_{\bbA}x)=\bbA(-,{\sup}_{\bbA}T_{\bbA}x)=\bbA(-,x)=\sY_{\bbA}x,$$
which completes the proof.
\end{proof}

\section{$\RQDist$ is dually equivalent to $\QCD$} \label{RQDist_QCD}

In this section we show that Theorem \ref{Kphi_ccd} can be enhanced to a (dual) equivalence of categories, with regular $\CQ$-distributors and (ccd) $\CQ$-categories as objects, respectively. To this end, first we establish the (contravariant) functoriality of the assignment $\phi\mapsto\Kphi$ by sending each commutative square
$$\bfig
\square<500,400>[\bbA`\bbA'`\bbB`\bbB';\zeta`\phi`\psi`\eta]
\place(250,0)[\circ] \place(250,400)[\circ] \place(0,200)[\circ] \place(500,200)[\circ]
\efig$$
in $\QDist$ to the left adjoint $\CQ$-functor
$$\CK(\zeta,\eta):=(\Kpsi\ \to/^(->/\PB'\to^{\eta^*}\PB\to^{\phi_*\phi^*}\Kphi).$$

\begin{prop} \label{K_functorial}
$\CK:\Arr(\QDist)^{\op}\lra\QSup$ is a quantaloid homomorphism.
\end{prop}

\begin{proof}
{\bf Step 1.} $\eta^*\psi_*\psi^*\leq\phi_*\phi^*\eta^*$. Consider the following diagram:
$$\bfig
\square|alrb|[\PB'`\PA'`\PB`\PA;\psi^*`\eta^*`\zeta^*`\phi^*]
\square(500,0)/>``>`>/[\PA'`\PB'`\PA`\PB;\psi_*``\eta^*`\phi_*]
\place(750,280)[\twoar(-1,-1)]
\efig$$
Note that the commutativity of the left square follows trivially from the functoriality of $(-)^*:(\QDist)^{\op}\lra\QCat$, and thus it suffices to check $\eta^*\psi_*\leq\phi_*\zeta^*$. Indeed, for all $\mu'\in\PA'$,
\begin{align*}
\eta^*\psi_*\mu'&=(\mu'\lda\psi)\circ\eta\\
&\leq((\mu'\lda\psi)\circ\eta\circ\phi)\lda\phi\\
&=((\mu'\lda\psi)\circ\psi\circ\zeta)\lda\phi\\
&\leq(\mu'\circ\zeta)\lda\phi\\
&=\phi_*\zeta^*\mu',
\end{align*}
as desired.

{\bf Step 2.} $\CK(\zeta,\eta):\Kpsi\lra\Kphi$ is a left adjoint $\CQ$-functor. For this, note that $\eta_*:\PB\lra\PB'$ can be restricted as a $\CQ$-functor $\eta_*:\Kphi\lra\Kpsi$. Indeed, from $\eta^*\dv\eta_*:\PB\lra\PB'$ and Step 1 one has $\psi_*\psi^*\leq\eta_*\phi_*\phi^*\eta^*$, which implies
$$\psi_*\psi^*\eta_*\lam\leq\eta_*\phi_*\phi^*\eta^*\eta_*\lam\leq\eta_*\phi_*\phi^*\lam=\eta_*\lam,$$
and consequently $\eta_*\lam\in\Kpsi$ for all $\lam\in\Kphi$. Now it remains to prove $\CK(\zeta,\eta)\dv\eta_*:\Kphi\lra\Kpsi$. Since $\phi_*\phi^*$ is a $\CQ$-monad on $\PB$,  it holds that
$$\PB(\eta^*\lam',\lam)\leq\PB(\phi_*\phi^*\eta^*\lam',\phi_*\phi^*\lam)=\PB(\phi_*\phi^*\eta^*\lam',\lam)\leq\PB(\eta^*\lam',\lam)$$
for all $\lam'\in\Kpsi$, $\lam\in\Kphi$. Hence
$$\Kphi(\CK(\zeta,\eta)\lam',\lam)=\PB(\phi_*\phi^*\eta^*\lam',\lam)=\PB(\eta^*\lam',\lam)=\PB'(\lam',\eta_*\lam)=\Kpsi(\lam',\eta_*\lam),$$
showing that $\CK(\zeta,\eta)\dv\eta_*:\Kphi\lra\Kpsi$.

{\bf Step 3.} $\CK:\Arr(\QDist)^{\op}\lra\QSup$ is a functor. For this one must prove
$$\CK(\zeta,\eta)\CK(\zeta',\eta')=\CK(\zeta'\circ\zeta,\eta'\circ\eta):\CK\xi\lra\Kphi$$
for any arrows $(\zeta,\eta):\phi\lra\psi$ and $(\zeta',\eta'):\psi\lra\xi$ in $\Arr(\QDist)$. It suffices to check that
\begin{equation} \label{phi_star_psi_eta}
\phi_*\phi^*\eta^*\psi_*\psi^*\eta'^*=\phi_*\phi^*\eta^*\eta'^*.
\end{equation}
On one hand, Step 1 implies $\phi_*\phi^*\eta^*\psi_*\psi^*\eta'^*\leq\phi_*\phi^*\phi_*\phi^*\eta^*\eta'^*=\phi_*\phi^*\eta^*\eta'^*$ since $\phi_*\phi^*$ is idempotent. On the other hand, $\phi_*\phi^*\eta^*\eta'^*\leq\phi_*\phi^*\eta^*\psi_*\psi^*\eta'^*$ holds trivially since $1_{\PB'}\leq\psi_*\psi^*$.

{\bf Step 4.} $\CK:\Arr(\QDist)^{\op}\lra\QSup$ is a quantaloid homomorphism. To show that $\CK$ preserves joins of arrows in $\Arr(\QDist)$, let $\{(\zeta_i,\eta_i)\}_{i\in I}$ be a family of arrows from $\phi:\bbA\oto\bbB$ to $\psi:\bbA'\oto\bbB'$ in $\Arr(\QDist)$, and one needs to verify
$$\CK\Big(\bv_{i\in I}\zeta_i,\bv_{i\in I}\eta_i\Big)=\bigsqcup_{i\in I}\CK(\zeta_i,\eta_i):\Kpsi\lra\Kphi,$$
where $\bigsqcup$ denotes the pointwise join in $\QSup(\Kpsi,\Kphi)$ inherited from $\Kphi$. Indeed, since $\phi_*\phi^*:\PB\lra\Kphi$ is a left adjoint $\CQ$-functor by Proposition \ref{monad_reflective}(1), one has
\begin{align*}
\CK\Big(\bv_{i\in I}\zeta_i,\bv_{i\in I}\eta_i\Big)\lam'&=\phi_*\phi^*\Big(\bv_{i\in I}\eta_i\Big)^*\lam'\\
&=\phi_*\phi^*\Big(\lam'\circ\bv_{i\in I}\eta_i\Big)\\
&=\phi_*\phi^*\Big(\bv_{i\in I}\lam'\circ\eta_i\Big)\\
&=\bigsqcup_{i\in I}\phi_*\phi^*(\lam'\circ\eta_i)&(\text{Proposition \ref{la_sup_preserving}(iii)})\\
&=\bigsqcup_{i\in I}\phi_*\phi^*\eta_i^*\lam'\\
&=\bigsqcup_{i\in I}\CK(\zeta_i,\eta_i)\lam'
\end{align*}
for all $\lam'\in\Kpsi$. This completes the proof.
\end{proof}

\begin{rem}
A \emph{Chu transform} (called an \emph{infomorphism} in \cite{Shen2014,Shen2013a})
$$(F,G):(\phi:\bbA\oto\bbB)\lra(\psi:\bbA'\oto\bbB')$$
between $\CQ$-distributors is a pair of $\CQ$-functors $F:\bbA\lra\bbA'$ and $G:\bbB'\lra\bbB$ such that $\psi(F-,-)=\phi(-,G-)$; or equivalently, $\psi\circ F_{\nat}=G^{\nat}\circ\phi$. $\CQ$-distributors and Chu transforms constitute a category $\QChu$ (denoted by $\QInfo$ in \cite{Shen2014,Shen2013a}), and one has a natural functor
$$(\Box_{\nat},\Box^{\nat}):\QChu\lra\Arr(\QDist),\quad(F,G)\mapsto(F_{\nat},G^{\nat}).$$
It should be pointed out that the composite of $\CK:\Arr(\QDist)^{\op}\lra\QSup$ with the functor $(\Box_{\nat},\Box^{\nat})^{\op}:(\QChu)^{\op}\lra\Arr(\QDist)^{\op}$ is exactly the functor $\CK:(\QChu)^{\op}\lra\QSup$ obtained in \cite{Shen2013a}. So, the functor $\CK$ given in Proposition \ref{K_functorial} is an extension of the functor $\CK:(\QChu)^{\op}\lra\QSup$ in \cite{Shen2013a}.
\end{rem}

\begin{prop} \label{KD_faithful}
For arrows $(\zeta,\eta),(\zeta',\eta'):(\phi:\bbA\oto\bbB)\lra(\psi:\bbA'\oto\bbB')$ in $\Arr(\QDist)$,
$$(\zeta,\eta)\sim(\zeta',\eta')\iff\CK(\zeta,\eta)=\CK(\zeta',\eta'):\Kpsi\lra\Kphi.$$
Therefore, $\CK$ factors uniquely through the quotient homomorphism $\Arr(\QDist)^{\op}\lra\BD(\QDist)^{\op}$ via a unique faithful quantaloid homomorphism $\KD:\BD(\QDist)^{\op}\lra\QSup$.
$$\bfig
\qtriangle/->`->`-->/<1500,500>[\Arr(\QDist)^{\op}`\BD(\QDist)^{\op}`\QSup;`\CK`\KD]
\efig$$
\end{prop}

\begin{proof}
If $(\zeta,\eta)\sim(\zeta',\eta')$, then $\eta\circ\phi=\eta'\circ\phi$, and thus the functoriality of $(-)^*:(\QDist)^{\op}\lra\QCat$ implies
$$\CK(\zeta,\eta)\lam'=\phi_*\phi^*\eta^*\lam'=\phi_*\phi^*\eta'^*\lam'=\CK(\zeta',\eta')\lam'$$
for all $\lam'\in\Kpsi$. Conversely, suppose $\CK(\zeta,\eta)=\CK(\zeta',\eta')$, one needs to show that $\eta\circ\phi=\eta'\circ\phi$. Indeed,
\begin{align*}
\eta(-,y')\circ\phi&=\phi^*\eta^*\sY_{\bbB'}y'&(\text{Equation \eqref{tphi_def}})\\
&=\phi^*\phi_*\phi^*\eta^*\sY_{\bbB'}y'&(\phi^*\dv\phi_*)\\
&=\phi^*\phi_*\phi^*\eta^*\psi_*\psi^*\sY_{\bbB'}y'&(\text{similar to the proof of Equation \eqref{phi_star_psi_eta}})\\
&=\phi^*\CK(\zeta,\eta)\psi_*\psi^*\sY_{\bbB'}y'\\
&=\phi^*\CK(\zeta',\eta')\psi_*\psi^*\sY_{\bbB'}y'&(\CK(\zeta,\eta)=\CK(\zeta',\eta'))\\
&=\eta'(-,y')\circ\phi&(\text{repeating the above steps})
\end{align*}
for all $y'\in\bbB'_0$, as desired.
\end{proof}

Let $\RArrQDist$ and $\QCD$ denote the full subquantaloids of $\Arr(\QDist)$ and $\QSup$, respectively, with objects restricting to regular $\CQ$-distributors and (ccd) $\CQ$-categories. Then $\CK$ sends objects in $\RArrQDist$ to those in $\QCD$ by Theorem \ref{Kphi_ccd}, and furthermore:

\begin{prop} \label{K_full_regular}
If $\psi$ is a regular $\CQ$-distributor, then the map
$$\CK:\Arr(\QDist)(\phi,\psi)\lra\QSup(\Kpsi,\Kphi)$$
is full. Hence, the restriction $\RArrQDist^{\op}\lra\QCD$ of $\CK:\Arr(\QDist)^{\op}\lra\QSup$ is a full quantaloid homomorphism.
\end{prop}

\begin{proof}
Given a left adjoint $\CQ$-functor $F:\Kpsi\lra\Kphi$, define a $\CQ$-distributor $\xi:\bbB\oto\bbB'$ through its transpose as
\begin{equation} \label{txi_def}
\txi:=(\bbB'\to^{\sY_{\bbB'}}\PB'\to^{\psi_*\psi^*}\Kpsi\to^F\Kphi\ \to/^(->/\PB).
\end{equation}
Since the square
$$\bfig
\square<1000,500>[\bbA`\bbA'`\bbB`\bbB';\zeta:=\opsi\circ\xi\circ\phi`\phi`\psi`\eta:=\psi\circ\opsi\circ\xi]
\place(500,0)[\circ] \place(0,250)[\circ] \place(1000,250)[\circ] \place(500,500)[\circ]
\efig$$
is clearly commutative, i.e., $(\zeta,\eta)\in\Arr(\QDist)(\phi,\psi)$, it now remains to prove $\CK(\zeta,\eta)=F$.

First, it follows from Proposition \ref{PA_PdA_sup}(1) and Equation \eqref{tensor_closure_system} that tensors in $\Kphi$ are given by
\begin{equation} \label{tensor_Kphi}
f\otimes\lam=\phi_*\phi^*(f\circ\mu)
\end{equation}
for all $\mu\in\Kphi$, $f\in\CP(t\mu)$.

Second, $\phi_*\phi^*:\PA\lra\Kphi$ and $\psi_*\psi^*:\PA'\lra\Kpsi$ are both left adjoint $\CQ$-functors (see Proposition \ref{monad_reflective}(1)), and thus so is the composite $F\psi_*\psi^*:\PA'\lra\Kphi$. Hence, for all $\lam'\in\Kpsi$ one has
\begin{align*}
\CK(\zeta,\eta)\lam'&=\phi_*\phi^*\eta^*\lam'\\
&=\phi_*\phi^*(\lam'\circ\psi\circ\opsi\circ\xi)&(\eta=\psi\circ\opsi\circ\xi)\\
&=\phi_*\phi^*\Big(\bv_{y'\in\bbB'_0}(\lam'\circ\psi\circ\opsi)(y')\circ\txi y'\Big)\\
&=\phi_*\phi^*\Big(\bv_{y'\in\bbB'_0}(\lam'\circ\psi\circ\opsi)(y')\circ(F\psi_*\psi^*\sY_{\bbB'}y')\Big)&(\text{Equation \eqref{txi_def}})\\
&=\bigsqcup_{y'\in\bbB'_0}\phi_*\phi^*((\lam'\circ\psi\circ\opsi)(y')\circ(F\psi_*\psi^*\sY_{\bbB'}y'))&(\text{Proposition \ref{la_sup_preserving}(iii)})\\
&=\bigsqcup_{y'\in\bbB'_0}(\lam'\circ\psi\circ\opsi)(y')\otimes(F\psi_*\psi^*\sY_{\bbB'}y')&(\text{Equation \eqref{tensor_Kphi}})\\
&=F\psi_*\psi^*\Big(\bv_{y'\in\bbB'_0}(\lam'\circ\psi\circ\opsi)(y')\circ(\sY_{\bbB'}y')\Big)&(\text{Proposition \ref{la_sup_preserving}(iii)})\\
&=F\psi_*\psi^*(\lam'\circ\psi\circ\opsi)\\
&=F\psi_*(\psi\circ\opsi\circ\psi)^*\lam'\\
&=F\psi_*\psi^*\lam'&(\psi\ \text{is regular})\\
&=F\lam',&(\lam'\in\Kpsi)
\end{align*}
where $\bv$ and $\bigsqcup$ respectively denote the underlying joins in $\PB$ and $\Kphi$, completing the proof.
\end{proof}

With $\RQDist$ denoting the full subquantaloid of $\BD(\QDist)$ consisting of regular $\CQ$-distributors, the combination of Propositions \ref{KD_faithful} and \ref{K_full_regular} shows that the restriction of $\KD:\BD(\QDist)^{\op}\lra\QSup$ on $\RQDist^{\op}$, which we denote by
$$\KR:\RQDist^{\op}\lra\QCD,$$
is a fully faithful quantaloid homomorphism. It is moreover essentially surjective as the second part of Theorem \ref{Kphi_ccd} claims, and thus the following conclusion arises as an immediate consequence:

\begin{thm} \label{KR_equiv}
$\KR:\RQDist^{\op}\lra\QCD$ is an equivalence of quantaloids. Hence, $\RQDist$ and $\QCD$ are dually equivalent quantaloids.
\end{thm}

\begin{rem}
The combination of Proposition \ref{RQ_equiv_IdmQ} and Theorem \ref{KR_equiv} in fact reproduces the dual equivalence
$$\BD(\QDist)_{\rm idm}^{\op}\simeq\QCD$$
which was revealed by Stubbe in \cite{Stubbe2007}; that is, the split-idempotent completion of $\QDist$ is dually equivalent to $\QCD$. In particular, the above equivalence reduces to the classical result
\begin{equation} \label{RWequiv}
\BD(\Rel)_{\rm idm}\simeq\Sup_{\rm ccd}
\end{equation}
of Rosebrugh and Wood by setting $\CQ={\bf 2}$ \cite{Rosebrugh1994}; just note that $\BD({\bf 2}\text{-}\Dist)_{\rm idm}$ is equivalent to its full subquantaloid $\BD(\Rel)_{\rm idm}$ and $\Sup_{\rm ccd}$ is self-dual. We also refer to Hofmann's work \cite{Hofmann2013} for another extensive study of the equivalence \eqref{RWequiv} in the context of $(\bbT,\sV)$-categories.
\end{rem}

We point out that one cannot establish Proposition \ref{K_full_regular} without assuming the regularity of $\psi$:

\begin{exmp} \label{K_not_full}
Let $\CQ=([0,\infty]^{\op},+)$ be Lawvere's quantale \cite{Lawvere1973} with the internal hom given by $x\ra y=\max\{0,y-x\}$\footnote{To avoid confusion, the symbols $\max$, $\leq$, etc. between (extended) real numbers always refer to the standard order, although the quantale $([0,\infty]^{\op},+)$ is equipped with the reverse order of real numbers.} for all $x,y\in Q$. A $\CQ$-distributor $\phi:\star\oto\star$ between singleton $\CQ$-categories is precisely an element $\phi\in[0,\infty]$, and it is not difficult to observe the following facts:
\begin{enumerate}[label={\rm(\arabic*)}]
\item $\phi$ is regular if and only if either $\phi=0$ or $\phi=\infty$.
\item $\Kphi=\CQ$ for all $\phi\in[0,\infty)$, where $\CQ$ is equipped the standard $\CQ$-category structure given by the internal hom $\ra$.
\end{enumerate}
We claim that the map
$$\CK:\Arr(\QDist)((0:\star\oto\star),(\psi:\star\oto\star))\lra\QSup(\CQ,\CQ)$$
is not full for any $\psi\in(0,\infty)$. Indeed, $\CK$ sends each $(\zeta,\eta):(0:\star\oto\star)\lra(\psi:\star\oto\star)$ to the left adjoint $\CQ$-functor
$$\eta^*:\CQ\lra\CQ,\quad x\mapsto\eta+x.$$
Taking any $\eta\in(0,\psi)$, $\eta^*:\CQ\lra\CQ$ defined as above is a left adjoint $\CQ$-functor, but there exists no $\zeta:\star\oto\star$ such that $(\zeta,\eta):(0:\star\oto\star)\lra(\psi:\star\oto\star)$ is an arrow in $\Arr(\QDist)$: one must have $\zeta=\eta-\psi$ to get a commutative square in $\QDist$ as illustrated below, which cannot be achieved when $\eta<\psi$.
$$\bfig
\square<800,500>[\star`\star`\star`\star;\zeta=\eta-\psi`0`\psi`\eta]
\place(400,0)[\circ] \place(400,500)[\circ] \place(0,250)[\circ] \place(800,250)[\circ]
\efig$$
\end{exmp}

\section{{\opccd} implies regularity} \label{opccd_regularity}

Theorem \ref{Kphi_ccd} asserts that $\Kphi$ is a (ccd) $\CQ$-category if $\phi$ is a regular $\CQ$-distributor, and the converse statement is known to be true when $\CQ={\bf 2}$ (see Theorem \ref{regular_cd_classical_ccd}). It is natural to ask whether one can establish Theorem \ref{regular_cd_classical_ccd} in the general setting. Unfortunately, the following counterexample gives a negative answer:

\begin{exmp} \label{Kphi_ccd_phi_not_regular}
For Lawvere's quantale $\CQ=([0,\infty]^{\op},+)$ considered in Example \ref{K_not_full}, any $\phi:\star\oto\star$ with $\phi\in(0,\infty)$ is a non-regular $\CQ$-distributor, but $\Kphi=\CQ=\CP\star$ is always a (ccd) $\CQ$-category.
\end{exmp}

In fact, it is the constructive complete \emph{co}distributivity (instead of constructive complete distributivity) of $\Kphi$ that necessarily implies the regularity of $\phi$, which fails to reveal itself in the case $\CQ={\bf 2}$ since these two notions are equivalent there. This observation will be demonstrated in Theorem \ref{Kphi_coccd_phi_regular} below.

A $\CQ$-category $\bbA$ is \emph{constructively completely codistributive}, or \emph{\opccd} for short, if $\bbA^{\op}$ is a (ccd) $\CQ^{\op}$-category. Explicitly, $\bbA$ is {\opccd} if it is complete and $\inf_{\bbA}:\PdA\lra\bbA$ admits a right adjoint $S_{\bbA}:\bbA\lra\PdA$ in $\QCat$; or equivalently, if there exists a string of adjoint $\CQ$-functors
$$\sYd_{\bbA}\dv{\inf}_{\bbA}\dv S_{\bbA}:\bbA\lra\PdA.$$

\begin{exmp} \label{PdA_coccd}
For any $\CQ$-category $\bbA$, $\PdA$ is an {\opccd} $\CQ$-category since $(\PdA)^{\op}\cong\PA^{\op}$ is a (ccd) $\CQ^{\op}$-category. Indeed, from \eqref{Fra_def} and Proposition \ref{PA_PdA_sup}(6) one sees that $\inf_{\PdA}=(\sYd_{\bbA})^{\nla}:\CPd\PdA\lra\PdA$ has a right adjoint $(\sYd_{\bbA})^{\nra}:\PdA\lra\CPd\PdA$ in $\QCat$.
\end{exmp}

The main result of this section is:

\begin{thm} \label{Kphi_coccd_phi_regular}
If $\Kphi$ is {\opccd}, then $\phi$ is regular.
\end{thm}

Before proving this theorem, we need some preparations. Each set $A$ equipped with a type map $t:A\lra\ob\CQ$ (called a \emph{$\CQ$-typed set}) may be viewed as a \emph{discrete} $\CQ$-category with
$$A(x,y)=\begin{cases}
1_{tx} & \text{if}\ x=y,\\
\bot_{tx,ty} & \text{else}.
\end{cases}$$
A \emph{$\CQ$-relation} (also \emph{$\CQ$-matrix} \cite{Betti1983,Heymans2010}) between $\CQ$-typed sets is exactly a $\CQ$-distributor between discrete $\CQ$-categories, and we write $|\phi|:|\bbA|\oto|\bbB|$ for the underlying $\CQ$-relation of a $\CQ$-distributor $\phi:\bbA\oto\bbB$. It is easy to verify the following lemma:

\begin{lem} \label{Q_rel_dist}
For $\CQ$-categories $\bbA$, $\bbB$ and any $\CQ$-relation $\phi:|\bbA|\oto|\bbB|$, the following statements are equivalent:
\begin{enumerate}[label={\rm(\arabic*)}]
\item $\phi:\bbA\oto\bbB$ is a $\CQ$-distributor.
\item $\phi\circ\bbA\leq\phi$ and $\bbB\circ\phi\leq\phi$.
\item $\bbA\leq\phi\rda\phi$ and $\bbB\leq\phi\lda\phi$.
\end{enumerate}
\end{lem}

\begin{lem} \label{Kphi_invariant}
$\Kphi=\CK|\phi|$.
\end{lem}

\begin{proof}
It suffices to prove that $\lam\in\CK|\phi|$ implies $\lam\in\PB$ for any $\lam\in\CP|\bbB|$. Indeed,
\begin{align*}
\lam\circ\bbB&=((\lam\circ|\phi|)\lda|\phi|)\circ\bbB&(\lam\in\CK|\phi|)\\
&\leq((\lam\circ|\phi|)\lda|\phi|)\circ(|\phi|\lda|\phi|)&(\text{Lemma \ref{Q_rel_dist}(iii)})\\
&\leq((\lam\circ|\phi|)\lda|\phi|)\\
&=\lam,&(\lam\in\CK|\phi|)
\end{align*}
showing that $\lam\in\PB$.
\end{proof}

\begin{lem} \label{interior_sup_inf}
If $F$ is a $\CQ$-comonad on $\PA$ and $\bbB:=\Fix(F)$, then
$${\sup}_{\bbB}\Theta=\Theta\circ\ga\quad\text{and}\quad{\inf}_{\bbB}\Lam=f(\Lam\rda\ga)$$
for all $\Theta\in\CP\bbB$, $\Lam\in\CPd\bbB$, where $\ga:\bbA\oto\bbB$ is the codomain restriction of the graph of the Yoneda embedding $(\sY_{\bbA})_{\nat}:\bbA\oto\PA$.
\end{lem}

\begin{proof}
Let $J:\bbB\lra\PA$ be the inclusion $\CQ$-functor, then
\begin{equation} \label{ga_Yoneda}
\ga=(\sY_{\bbA})_{\nat}(-,J-)=\PA(-,J-)\circ(\sY_{\bbA})_{\nat}=J^{\nat}\circ(\sY_{\bbA})_{\nat}.
\end{equation}
Note that $\Theta\circ J^{\nat}=J^{\ra}\Theta\in\CP\PA$, and thus
\begin{align*}
\bbB(\Theta\circ\ga,\mu)&=\bbB(\Theta\circ J^{\nat}\circ(\sY_{\bbA})_{\nat},\mu)&(\text{Equation \eqref{ga_Yoneda}})\\
&=\PA(\Theta\circ J^{\nat}\circ(\sY_{\bbA})_{\nat},\mu)\\
&=\PA({\sup}_{\PA}J^{\ra}\Theta,\mu)&(\text{Proposition \ref{PA_PdA_sup}(3)})\\
&=\PB(\Theta,J^{\la}\sY_{\PA}\mu)&({\sup}_{\PA}J^{\ra}\dv J^{\la}\sY_{\PA}:\PA\lra\PB)\\
&=(\PA(-,\mu)\circ J_{\nat})\lda\Theta\\
&=\PA(J-,\mu)\lda\Theta\\
&=\bbB(-,\mu)\lda\Theta
\end{align*}
for all $\mu\in\bbB$, which indicates $\Theta\circ\ga={\sup}_{\bbB}\Theta$. The second identity may be verified similarly.
\end{proof}

\begin{proof}[of Theorem \ref{Kphi_coccd_phi_regular}]
Let $\phi:\bbA\oto\bbB$ be a $\CQ$-distributor. Note that $(\bbA_0,\phi\rda\phi)$ is also a $\CQ$-category whose underlying $\CQ$-typed set is $\bbA_0$ with hom-arrows given by $$(\phi\rda\phi)(x,x')=\phi(x',-)\rda\phi(x,-)\in\CQ(tx,tx')$$
for all $x,x'\in\bbA_0$; so is $(\bbB_0,\phi\lda\phi)$. Thus, from Lemma \ref{Q_rel_dist}(iii) one sees that $\phi:(\bbA_0,\phi\rda\phi)\oto(\bbB_0,\phi\lda\phi)$ is also a $\CQ$-distributor.

Moreover, Lemma \ref{Kphi_invariant} tells us that $\Kphi$ is independent of the $\CQ$-categorical structures of the domain and codomain of $\phi$, which guarantees
$$\CK(\phi:\bbA\oto\bbB)=\CK(|\phi|:|\bbA|\oto|\bbB|)=\CK(\phi:(\bbA_0,\phi\rda\phi)\oto(\bbB_0,\phi\lda\phi)).$$
Hence, without loss of generality one may assume the $\CQ$-distributor $\phi:\bbA\oto\bbB$ to satisfy $\bbA=\phi\rda\phi$ and $\bbB=\phi\lda\phi$; otherwise one just needs to consider the $\CQ$-distributor $\phi:(\bbA_0,\phi\rda\phi)\oto(\bbB_0,\phi\lda\phi)$ instead, because the regularity of $\phi:\bbA\oto\bbB$ obviously follows from that of $\phi:(\bbA_0,\phi\rda\phi)\oto(\bbB_0,\phi\lda\phi)$.

Since $\Rphi(\cong\Kphi)$ is {\opccd}, one has an adjunction
$${\inf}_{\Rphi}\dv S_{\Rphi}:\Rphi\lra\CPd\Rphi$$
in $\QCat$. Thus $\inf_{\Rphi}:\CPd\Rphi\lra\Rphi$ is $\sup$-preserving (see Proposition \ref{la_sup_preserving}(ii)) in the sense that
\begin{equation} \label{inf_Rphi_preserves_sup}
{\inf}_{\Rphi}{\sup}_{\CPd\Rphi}\Psi={\sup}_{\Rphi}{\inf}_{\Rphi}^{\ra}\Psi
\end{equation}
for all $\Psi\in\CP\CPd\Rphi$. We shall apply \eqref{inf_Rphi_preserves_sup} to the presheaf
\begin{equation} \label{Psi_def}
\Psi=\phi(-,b)\circ H^{\nat}:\CPd\Rphi\oto\bbA\oto\star_{tb}
\end{equation}
on $\CPd\Rphi$ for any $b\in\bbB_0$, where the $\CQ$-functor $H:\bbA\lra\CPd\Rphi$ is given by
\begin{equation} \label{H_def}
H:=(\bbA\to^{\sY_{\bbA}}\PA\to^{\phi^*\phi_*}\Rphi\to^{S_{\Rphi}}\CPd\Rphi).
\end{equation}

{\bf Step 1.} Let $\ga:\bbA\oto\Rphi$ be the codomain restriction of $(\sY_{\bbA})_{\nat}:\bbA\oto\PA$. We claim
$${\inf}_{\Rphi}\hga=\phi^*\phi_*\sY_{\bbA}:\bbA\lra\Rphi.$$
$$\bfig
\square<800,500>[\bbA`\CPd\Rphi`\PA`\Rphi;\hga`\sY_{\bbA}`\inf_{\Rphi}`\phi^*\phi_*]
\efig$$
In fact, if one restricts the codomain of $\tphi:\bbB\lra\PA$ to $\Rphi$, then similar to Equation \eqref{dist_Yoneda} one derives
\begin{equation} \label{tphi_ga}
\phi=\tphi^{\nat}\circ\ga:\bbA\oto\Rphi\oto\bbB.
\end{equation}
It follows that
$$\bbA\leq\ga\rda\ga\leq(\tphi^{\nat}\circ\ga)\rda(\tphi^{\nat}\circ\ga)=\phi\rda\phi=\bbA,$$
and consequently $\bbA=\ga\rda\ga$. Therefore, by Lemma \ref{interior_sup_inf} one soon has
$${\inf}_{\Rphi}\hga x=\phi^*\phi_*(\hga x\rda\ga)=\phi^*\phi_*(\ga(x,-)\rda\ga)=\phi^*\phi_*(\bbA(-,x))=\phi^*\phi_*\sY_{\bbA}x$$
for all $x\in\bbA_0$, as desired.

{\bf Step 2.} ${\sup}_{\CPd\Rphi}\Psi\leq\ga\lda\phi(-,b)$ in $\QDist(\star_{tb},\Rphi)$. Indeed,
\begin{align*}
{\sup}_{\CPd\Rphi}\Psi&=(\sYd_{\Rphi})^{\nat}\lda\Psi&(\text{Proposition \ref{PA_PdA_sup}(5)})\\
&\leq((\sYd_{\Rphi})^{\nat}\circ\hga_{\nat})\lda(\Psi\circ\hga_{\nat})\\
&=\ga\lda(\Psi\circ\hga_{\nat})&(\text{Equation \eqref{dist_Yoneda}})\\
&=\ga\lda(\phi(-,b)\circ H^{\nat}\circ\hga_{\nat})&(\text{Equation \eqref{Psi_def}})\\
&=\ga\lda(\phi(-,b)\circ(\phi^*\phi_*\sY_{\bbA})^{\nat}\circ S_{\Rphi}^{\nat}\circ\hga_{\nat})&(\text{Equation \eqref{H_def}})\\
&=\ga\lda(\phi(-,b)\circ(\phi^*\phi_*\sY_{\bbA})^{\nat}\circ({\inf}_{\Rphi})_{\nat}\circ\hga_{\nat})&({\inf}_{\Rphi}\dv S_{\Rphi})\\
&=\ga\lda(\phi(-,b)\circ(\phi^*\phi_*\sY_{\bbA})^{\nat}\circ(\phi^*\phi_*\sY_{\bbA})_{\nat})&(\text{Step 1})\\
&\leq\ga\lda\phi(-,b).&((\phi^*\phi_*\sY_{\bbA})_{\nat}\dv(\phi^*\phi_*\sY_{\bbA})^{\nat})
\end{align*}

{\bf Step 3.} ${\inf}_{\Rphi}{\sup}_{\CPd\Rphi}\Psi\geq\phi(-,b)$. First, we prove
\begin{equation} \label{phib}
\phi=(\ga\lda\phi)\rda\ga.
\end{equation}
On one hand, if one restricts the codomain of $\tphi:\bbB\lra\PA$ to $\Rphi$, then
\begin{equation} \label{cograph_lda}
\tphi^{\nat}\circ(\ga\lda\phi)=(\tphi^{\nat}\circ\ga)\lda\phi.
\end{equation}
Indeed, from $\tphi_{\nat}\dv\tphi^{\nat}$ one has
\begin{align*}
(\tphi^{\nat}\circ\ga)\lda\phi&\leq\tphi^{\nat}\circ\tphi_{\nat}\circ((\tphi^{\nat}\circ\ga)\lda\phi)\\
&\leq\tphi^{\nat}\circ((\tphi_{\nat}\circ\tphi^{\nat}\circ\ga)\lda\phi)\\
&\leq\tphi^{\nat}\circ(\ga\lda\phi),
\end{align*}
and the reverse inequality is trivial. Thus
\begin{align*}
(\ga\lda\phi)\rda\ga&\leq(\tphi^{\nat}\circ(\ga\lda\phi))\rda(\tphi^{\nat}\circ\ga)\\
&=((\tphi^{\nat}\circ\ga)\lda\phi)\rda(\tphi^{\nat}\circ\ga)&(\text{Equation \eqref{cograph_lda}})\\
&=(\phi\lda\phi)\rda\phi&(\text{Equation \eqref{tphi_ga}})\\
&=\bbB\rda\phi&(\phi\lda\phi=\bbB)\\
&=\phi.
\end{align*}
On the other hand, $\phi\leq(\ga\lda\phi)\rda\ga$ is trivial, showing the validity of \eqref{phib}.

Second, since it follows from ${\sup}_{\CPd\Rphi}\Psi\leq\ga\lda\phi(-,b)$ in $\QDist(\star_{tb},\Rphi)$ that ${\sup}_{\CPd\Rphi}\Psi\geq\ga\lda\phi(-,b)$ in $\CPd\Rphi$ (see Remark \ref{PdA_QDist_order}), it follows that
\begin{align*}
{\inf}_{\Rphi}{\sup}_{\CPd\Rphi}\Psi&\geq{\inf}_{\Rphi}(\ga\lda\phi(-,b))\\
&=\phi^*\phi_*((\ga\lda\phi(-,b))\rda\ga)&(\text{Lemma \ref{interior_sup_inf}})\\
&=\phi^*\phi_*(\phi(-,b))&(\text{Equation \eqref{phib}})\\
&=\phi^*\phi_*\phi^*\sY_{\bbB}b&(\text{Equation \eqref{tphi_def}})\\
&=\phi^*\sY_{\bbB}b&(\phi^*\dv\phi_*)\\
&=\phi(-,b).&(\text{Equation \eqref{tphi_def}})
\end{align*}

{\bf Step 4.} ${\sup}_{\Rphi}{\inf}_{\Rphi}^{\ra}\Psi=\phi(-,b)\circ(\bbA\lda\phi)\circ\phi$. First, ${\inf}_{\Rphi}S_{\Rphi}=1_{\Rphi}$ since
$$1_{\Rphi}={\inf}_{\Rphi}\sYd_{\Rphi}\leq{\inf}_{\Rphi}S_{\Rphi}{\inf}_{\Rphi}\sYd_{\Rphi}={\inf}_{\Rphi}S_{\Rphi}\leq 1_{\Rphi}$$
follows from $\sYd_{\Rphi}\dv{\inf}_{\Rphi}\dv S_{\Rphi}:\Rphi\lra\CPd\Rphi$.

Second,
\begin{equation} \label{A_lda_phi_circ_phi}
(\bbA\lda\phi)\circ\phi=(\phi^*\phi_*\sY_{\bbA})^{\nat}\circ\ga
\end{equation}
since
\begin{align*}
(\bbA(-,x)\lda\phi)\circ\phi&=\phi^*\phi_*\sY_{\bbA}x\\
&={\sup}_{\Rphi}\sY_{\Rphi}\phi^*\phi_*\sY_{\bbA}x&({\sup}_{\Rphi}\sY_{\Rphi}=1_{\Rphi})\\
&=(\sY_{\Rphi}\phi^*\phi_*\sY_{\bbA}x)\circ\ga&(\text{Lemma \ref{interior_sup_inf}})\\
&=\Rphi(-,\phi^*\phi_*\sY_{\bbA}x)\circ\ga\\
&=(\phi^*\phi_*\sY_{\bbA})^{\nat}(-,x)\circ\ga.
\end{align*}
for all $x\in\bbA_0$. Hence,
\begin{align*}
{\sup}_{\Rphi}{\inf}_{\Rphi}^{\ra}\Psi&={\sup}_{\Rphi}(\Psi\circ{\inf}_{\Rphi}^{\nat})&(\text{see \eqref{Fra_def}})\\
&={\sup}_{\Rphi}(\phi(-,b)\circ(\phi^*\phi_*\sY_{\bbA})^{\nat}\circ S_{\Rphi}^{\nat}\circ{\inf}_{\Rphi}^{\nat})&(\text{Equations \eqref{Psi_def} \& \eqref{H_def}})\\
&={\sup}_{\Rphi}(\phi(-,b)\circ(\phi^*\phi_*\sY_{\bbA})^{\nat})&({\inf}_{\Rphi}S_{\Rphi}=1_{\Rphi})\\
&=\phi(-,b)\circ(\phi^*\phi_*\sY_{\bbA})^{\nat}\circ\ga&(\text{Lemma \ref{interior_sup_inf}})\\
&=\phi(-,b)\circ(\bbA\lda\phi)\circ\phi.&(\text{Equation \eqref{A_lda_phi_circ_phi}})
\end{align*}

{\bf Step 5.} As $b\in\bbB_0$ is arbitrary, Step 3 and Step 4 in combination with Equation \eqref{inf_Rphi_preserves_sup} lead to
$$\phi\circ\ophi\circ\phi=\phi\circ((\phi\rda\phi)\lda\phi)\circ\phi=\phi\circ(\bbA\lda\phi)\circ\phi\geq\phi.$$
From Proposition \ref{regular_condition} one concludes that $\phi$ is regular.
\end{proof}

The following corollary is an immediate consequence of Theorems \ref{Kphi_ccd} and \ref{Kphi_coccd_phi_regular}:

\begin{cor} \label{Kphi_coccd_imply_ccd}
For any $\CQ$-distributor $\phi$, if $\Kphi$ is {\opccd}, then it is also (ccd).
\end{cor}

\begin{exmp}
For any small quantaloid $\CQ$, the terminal object in $\QCat$ is given by $(\ob\CQ,\top)$ with $\top(X,Y)=\top_{X,Y}:X\lra Y$, the top element in $\CQ(X,Y)$. $(\ob\CQ,\top)$ is both (ccd) and {\opccd} since
$$(\ob\CQ,\top)=\CP\varnothing=\CPd\varnothing,$$
where $\varnothing$ denotes the empty $\CQ$-category, i.e., the initial object in $\QCat$. Therefore:
\begin{enumerate}[label={\rm(\arabic*)}]
\item $\CK\varnothing(=\CP\varnothing)$ is both (ccd) and {\opccd}.
\item For any $X,Y\in\ob\CQ$, the bottom element $\bot_{X,Y}\in\CQ(X,Y)$ gives a $\CQ$-distributor $\bot_{X,Y}:\star_X\oto\star_Y$. Then $\CK\bot_{X,Y}$ is a $\CQ$-subcategory of $\CP Y$ consisting of $\CQ$-arrows $\top_{Y,Z}$ $(Z\in\ob\CQ)$. It is straightforward to verify $\CK\bot_{X,Y}\cong(\ob\CQ,\top)$ and, consequently, $\CK\bot_{X,Y}$ is both (ccd) and {\opccd}.
\end{enumerate}
\end{exmp}

It should be reminded that for a general quantaloid $\CQ$, neither the converse statement of Theorem \ref{Kphi_coccd_phi_regular} nor that of Corollary \ref{Kphi_coccd_imply_ccd} is true:

\begin{exmp} \label{phi_regular_Kphi_not_opccd}
For the identity $\CQ$-distributor on any $\CQ$-category $\bbA$ (which is clearly regular), $\CK\bbA=\PA$ is (ccd) (see Examples \ref{KA} and \ref{PA_ccd}) but in general not {\opccd}.
\end{exmp}

However, one is able to reconcile the notions of regularity, (ccd) and {\opccd} when $\CQ$ is a Girard quantaloid as discussed in the next section.

\section{Girard quantaloids reconcile regularity, (ccd) and {\opccd}} \label{Girard_quantaloids_reg_ccd_opccd}

In a quantaloid $\CQ$, a family of $\CQ$-arrows $\{d_X:X\lra X\}_{X\in\ob\CQ}$ is a \emph{cyclic family} (resp. \emph{dualizing family}) if
$$d_X\lda f=f\rda d_Y\quad (\text{resp.}\ (d_X\lda f)\rda d_X=f=d_Y\lda(f\rda d_Y))$$
for all $\CQ$-arrows $f:X\lra Y$. A \emph{Girard quantaloid} \cite{Rosenthal1992} is a quantaloid $\CQ$ equipped with a cyclic dualizing family of $\CQ$-arrows.

In a Girard quantaloid $\CQ$, each $\CQ$-arrow $f:X\lra Y$ has a \emph{complement}
$$\neg f=d_X\lda f=f\rda d_Y:Y\lra X,$$
which clearly satisfies $\neg\neg f=f$. For each $\CQ$-category $\bbA$,
$$(\neg\bbA)(y,x)=\neg\bbA(x,y)$$
gives a $\CQ$-distributor $\neg\bbA:\bbA\oto\bbA$, and it is straightforward to check that
$$\{\neg\bbA:\bbA\oto\bbA\}_{\bbA\in\ob(\QDist)}$$
is a cyclic dualizing family of $\QDist$; this gives the ``only if'' part of the following proposition. As for the ``if'' part, just note that $\CQ$ can be fully faithfully embedded in $\QDist$:

\begin{prop} (See \cite{Rosenthal1992}.) \label{QDist_Girard}
A small quantaloid $\CQ$ is a Girard quantaloid if, and only if, $\QDist$ is a Girard quantaloid.
\end{prop}

Hence, with $\CQ$ being Girard, each $\CQ$-distributor $\phi:\bbA\oto\bbB$ has a \emph{complement}
$$\neg\phi:=\neg\bbA\lda\phi=\phi\rda\neg\bbB:\bbB\oto\bbA.$$

\begin{exmp}
\begin{enumerate}[label={\rm(\arabic*)}]
\item Every Girard quantale \cite{Rosenthal1990,Yetter1990} is a one-object Girard quantaloid.
\item $\Rel$ is a Girard quantaloid since it is a full subquantaloid of the Girard quantaloid ${\bf 2}\text{-}\Dist$, where ${\bf 2}\text{-}\Dist$ being Girard follows from Proposition \ref{QDist_Girard} and the fact that ${\bf 2}$ is a Girard quantale.
\item Each complete Boolean algebra $(L,\wedge,\vee,\neg,0,1)$ induces a Girard quantaloid $\DL$ \cite{Hohle2011,Pu2012,Walters1981} (i.e., the quantaloid of diagonals in the one-object quantaloid $L$) with the following data:
\begin{itemize}
\item objects in $\DL$ are the elements of $L$;
\item $\DL(X,Y)=\{f\in L:f\leq X\wedge Y\}$ with inherited order from $L$;
\item the composition of $\DL$-arrows $f\in\DL(X,Y)$, $g\in\DL(Y,Z)$ is given by $g\circ f=g\wedge f$;
\item the identity $\DL$-arrow in $\DL(X,X)$ is $X$ itself.
\end{itemize}
It is straightforward to check that $\{0:X\lra X\}_{X\in L}$ is a cyclic dualizing family in $\DL$.
\item Each quantaloid $\CQ$ is embedded in a Girard quantaloid $\CQ_G$ \cite{Shen2016a} with the following data:
\begin{itemize}
\item objects in $\CQ_G$ are the same as those in $\CQ$;
\item $\CQ_G(X,Y)=\CQ(X,Y)\times\CQ(Y,X)$ with
$$\bv_{i\in I}(f_i,f'_i)=\Big(\bv_{i\in I}f_i,\bw_{i\in I}f'_i\Big)$$
for all $\{(f_i,f'_i)\}_{i\in I}\subseteq\CQ_G(X,Y)$;
\item the composition of $\CQ_G$-arrows $(f,f'):X\lra Y$, $(g,g'):Y\lra Z$ is given by
$$(g,g')\circ(f,f')=(g\circ f,\ (f'\lda g)\wedge(f\rda g'));$$
\item the identity $\CQ_G$-arrow in $\CQ_G(X,X)$ is $(1_X,\top_{X,X}):X\lra X$.
\end{itemize}
\end{enumerate}
\end{exmp}

The most important property of Girard-quantaloid-enriched categories is that presheaf $\CQ$-categories are isomorphic to copresheaf $\CQ$-categories:

\begin{prop} \label{PA_PdA_iso}
If $\CQ$ is a small Girard quantaloid, then for any $\CQ$-category $\bbA$,
$$\neg:\PA\lra\PdA$$
is an isomorphism in $\QCat$.
\end{prop}

\begin{proof}
Since $\{\neg\bbA\}_{\bbA\in\ob(\QDist)}$ is a cyclic dualizing family, one has
$$\PA(\mu,\lam)=\lam\lda\mu=((\neg\bbA\lda\lam)\rda\neg\bbA)\lda\mu=(\neg\bbA\lda\lam)\rda(\neg\bbA\lda\mu)=\PdA(\neg\mu,\neg\lam)$$
for all $\mu,\lam\in\PA$. Thus $\neg:\PA\lra\PdA$ is a fully faithful $\CQ$-functor, and consequently an isomorphism in $\QCat$ since it is obviously surjective.
\end{proof}

\begin{prop} \label{Girard_ccd=coccd}
If $\CQ$ is a small Girard quantaloid, then a $\CQ$-category is (ccd) if and only if it is {\opccd}.
\end{prop}

\begin{proof}
Suppose that $\bbA$ is a skeletal (ccd) $\CQ$-category. Then $\bbA$ is a retract of $\PA$ in $\QInf$ since $\bbA\two/->`<-/^{\sY_{\bbA}}_{\sup_{\bbA}}\PA$ satisfies $\sup_{\bbA}\sY_{\bbA}=1_{\bbA}$ and both $\sup_{\bbA}$, $\sY_{\bbA}$ are right adjoints in $\QCat$. As $\PA\cong\PdA$ by Proposition \ref{PA_PdA_iso}, $\PA$ is {\opccd}; hence the dual of Proposition \ref{retract_ccd} implies that $\bbA$ is {\opccd}. This proves the ``only if'' part, and the ``if'' part is precisely the dual of the ``only if'' part.
\end{proof}

Therefore, as an immediate consequence of Theorems \ref{Kphi_ccd}, \ref{Kphi_coccd_phi_regular} and Proposition \ref{Girard_ccd=coccd}, one has the following generalized version of Theorem \ref{regular_cd_classical_ccd}:

\begin{thm} \label{regular_ccd_Girard}
If $\CQ$ is a small Girard quantaloid, then for any $\CQ$-distributor $\phi$, the following statements are equivalent:
\begin{enumerate}[label={\rm(\roman*)}]
\item $\phi$ is regular.
\item $\Kphi$ is (ccd).
\item $\Kphi$ is {\opccd}.
\end{enumerate}
\end{thm}

\section{When $\CQ$ is a commutative integral quantale} \label{Q_quantale}

One may wonder whether $\CQ$ being Girard is essential for Theorem \ref{regular_ccd_Girard} to be true; that is, suppose that
$$\phi\ \text{is regular}\iff\Kphi\ \text{is (ccd)}\iff\Kphi\ \text{is {\opccd}}$$
for all $\CQ$-distributors $\phi$, is $\CQ$ necessarily a Girard quantaloid? Although it is difficult to answer this question for a general small quantaloid $\CQ$, we do have some partial results when $\CQ$ is a commutative integral quantale as the following Theorem \ref{ccd_coccd_Girard} shows. As a preparation, we explain the involved notions first.

An \emph{integral quantale} $(\CQ,\&)$ is a one-object quantaloid in which the unit $1$ of the underlying monoid $(\CQ,\&)$ is the top element of the complete lattice $\CQ$. It is moreover \emph{commutative} if $p\& q=q\& p$ for all $p,q\in\CQ$, and in this case we write
$$p\ra q:=q\lda p=p\rda q$$
for the adjoints induced by the monoid multiplication $\&$, which satisfies
$$p\& q\leq r\iff p\leq q\ra r$$
for all $p,q,r\in\CQ$. The operation $\ra$ makes $\CQ$ itself a $\CQ$-category, which may also be viewed as the presheaf $\CQ$-category of the singleton $\CQ$-category $\star$, i.e., $\CQ=\CP\star$.

A \emph{Girard quantale} \cite{Rosenthal1990,Yetter1990} is precisely a one-object Girard quantaloid. For a commutative integral quantale $(\CQ,\&)$, the commutativity ensures that every element of $\CQ$ is cyclic, and the integrality forces a dualizing element in $\CQ$, whenever it exists, to be the bottom element $\bot$ of $\CQ$. Hence, a commutative integral quantale $(\CQ,\&)$ is Girard if, and only if,
$$q=(q\ra\bot)\ra\bot$$
for all $q\in\CQ$.

\begin{exmp}
\begin{enumerate}[label={\rm(\arabic*)}]
\item Every frame is a commutative integral quantale, and Girard frames are precisely complete Boolean algebras.
\item Every complete BL-algebra \cite{Hajek1998} is a commutative integral quantale, and it is Girard if and only if it is an MV-algebra \cite{Chang1958}. In particular, the unit interval $[0,1]$ equipped with a continuous t-norm \cite{Klement2000} is a commutative integral quantale, and it becomes a Girard quantale if and only if it is isomorphic to $[0,1]$ equipped with the {\L}ukasiewicz t-norm.
\item Lawvere's quantale $\CQ=([0,\infty]^{\op},+)$ (see Example \ref{K_not_full}) is commutative and integral, but it is not Girard.
\end{enumerate}
\end{exmp}

\begin{thm} \label{ccd_coccd_Girard}
\footnote{The authors are indebted to Professor Dexue Zhang for helpful discussions on this theorem.}
Let $(\CQ,\&)$ be a commutative integral quantale. Then the following statements are equivalent:
\begin{enumerate}[label={\rm(\roman*)}]
\item $\bot$ is a dualizing element, hence $\CQ$ is a Girard quantale.
\item A $\CQ$-category is (ccd) if and only if it is {\opccd}.
\item For any $\CQ$-distributor $\phi$, $\Kphi$ is (ccd) if and only if it is {\opccd}.
\item For any $\CQ$-distributor $\phi$, $\phi$ is regular if and only if $\Kphi$ is {\opccd}.
\item $\CQ$ is an {\opccd} $\CQ$-category.
\end{enumerate}
\end{thm}

\begin{proof}
(i)${}\Lra{}$(ii): Proposition \ref{Girard_ccd=coccd}.

(ii)${}\Lra{}$(iii): Trivial.

(iii)${}\Lra{}$(iv): Theorems \ref{Kphi_ccd} and \ref{Kphi_coccd_phi_regular}.

(iv)${}\Lra{}$(v): Since the identity $\CQ$-distributor on the singleton $\CQ$-category $\star$ is regular, from Example \ref{KA} one sees that $\CQ=\CP\star=\CK\star$ is {\opccd}.

(v)${}\Lra{}$(i): Since $\CQ$ is {\opccd}, $\inf_{\CQ}:\CPd\CQ\lra\CQ$ is a left adjoint in $\QCat$. Let $\lam$ be the bottom element in $\QDist(\star,\CQ)$, i.e., $\lam(p)=\bot$ for all $p\in\CQ$. Then for any $q\in\CQ=\CP\star$,
\begin{align*}
q&=q\& 1\\
&=q\otimes_{\CQ} 1&(\text{Proposition \ref{PA_PdA_sup}(1)})\\
&=q\otimes_{\CQ}\bw_{p\in\CQ}\lam(p)\ra p&(\lam(p)=\bot)\\
&=q\otimes_{\CQ}{\inf}_{\CQ}\lam&(\text{Proposition \ref{PA_PdA_sup}(4)})\\
&={\inf}_{\CQ}(q\otimes_{\CPd\CQ}\lam)&(\text{Proposition \ref{la_sup_preserving}(iii)})\\
&={\inf}_{\CQ}(q\ra\lam)&(\text{Proposition \ref{PA_PdA_sup}(2)})\\
&=\bw_{p\in\CQ}(q\ra\lam(p))\ra p&(\text{Proposition \ref{PA_PdA_sup}(4)})\\
&=\bw_{p\in\CQ}(q\ra\bot)\ra p&(\lam(p)=\bot)\\
&=(q\ra\bot)\ra\bot,
\end{align*}
which shows that $\bot$ is a dualizing element, completing the proof.
\end{proof}

\section{Concluding remarks}

Let $\phi$ be a $\CQ$-distributor. Consider the implications labelled in the following diagram:
$$\bfig
\morphism|m|/=>/<1000,0>[\Kphi\ \text{is {\opccd}}`\phi\ \text{is regular};\tc{1}]
\morphism(1000,0)|m|/=>/<950,0>[\phi\ \text{is regular}`\Kphi\ \text{is (ccd)};\tc{2}]
\morphism(1950,0)|a|/{@{=>}@/^-2em/}/<-1950,0>[\Kphi\ \text{is (ccd)}`\Kphi\ \text{is {\opccd}};\tc{3}]
\morphism(1000,0)|b|/{@{=>}@/^2em/}/<-1000,0>[\phi\ \text{is regular}`\Kphi\ \text{is {\opccd}};\tc{4}]
\morphism(1950,0)|b|/{@{=>}@/^2em/}/<-950,0>[\Kphi\ \text{is (ccd)}`\phi\ \text{is regular};\tc{5}]
\efig
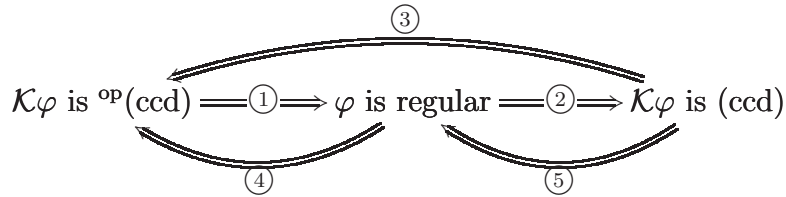
$$
\begin{itemize}
\item For any small quantaloid $\CQ$, \tc{1} and \tc{2} are always true (Theorems \ref{Kphi_ccd} and \ref{Kphi_coccd_phi_regular}).
\item If $\CQ$ is a Girard quantaloid, then \tc{1}-\tc{5} are all true (Theorem \ref{regular_ccd_Girard}).
\item When $\CQ$ is a commutative integral quantale, either \tc{3} or \tc{4} is true for all $\phi$ if, and only if, $\CQ$ is a Girard quantale (Theorem \ref{ccd_coccd_Girard}).
\end{itemize}

We end this paper with the following questions:
\begin{itemize}
\item When $\CQ$ is a commutative integral quantale, what is the necessary and sufficient condition for \tc{5} to be true for all $\phi$? We do not know the answer even in this special case.
\item For a general small quantaloid $\CQ$, what is the necessary and sufficient condition for any (or all) of \tc{3}, \tc{4}, \tc{5} to be true for all $\phi$?
\end{itemize}


\end{document}